\documentclass[10pt,twoside,openany]{article}
\usepackage{amsmath,amssymb,amsthm,bm}
\usepackage{cite}
\usepackage{graphicx}
\usepackage{subfigure}
\usepackage{enumerate}
\usepackage{bm}
\usepackage{threeparttable}
\usepackage{multirow}
\usepackage{booktabs}
\usepackage{graphicx}
\usepackage{epstopdf}
\usepackage{mathtools}
\usepackage{mathrsfs}
\usepackage{framed}
\usepackage{setspace}
\usepackage{color}
\usepackage{changepage}
\usepackage{tikz}
\newcommand*{\circled}[1]{\lower.7ex\hbox{\tikz\draw (0pt, 0pt)%
    circle (.5em) node {\makebox[1em][c]{\small #1}};}}

\newcommand{\normmm}[1]{{\left\vert\kern-0.25ex\left\vert\kern-0.25ex\left\vert #1
    \right\vert\kern-0.25ex\right\vert\kern-0.25ex\right\vert}}

\usepackage[colorlinks,
            linkcolor=blue,
            anchorcolor=blue,
            citecolor=red]{hyperref}
\allowdisplaybreaks[4]

\newtheorem{Definition}{Definition}[section]
\newtheorem{Lemma}{Lemma}[section]
\newtheorem{Theorem}{Theorem}[section]
\newtheorem{Remark}{Remark}[section]

\newcounter{saveeqn}%

\newcommand{\enabstractname}{Abstract}

\makeatletter
\evensidemargin
\oddsidemargin
\makeatother

\title{\bf  An efficient Chorin-Temam projection proper orthogonal decomposition based reduced-order model for nonstationary Stokes equations}
\author{
Xi Li\footnote{School of Mathematics, Sichuan University, Chengdu, Sichuan 610064, China (li\_xi@stu.scu.edu.cn). The work of this author was supported by the National Natural Science Foundation of China(Grant No. 11971337).},
\ Yan Luo\footnote{School of Mathematical Sciences, University of Electronic Science and Technology of China,
Chengdu, Sichuan 611731, China (luoyan\_16@126.com). The work of this author was supported by
the Young Scientists Fund of the National Natural Science Foundation of China(Grant No. 11901078) and
by the Fundamental Research Funds for the Central Universities(Grant No. ZYGX2020J021).}
\ and Minfu Feng\footnote{Corresponding author. School of Mathematics, Sichuan University, Chengdu, Sichuan 610064, China (fmf@scu.edu.cn). The work of this author was supported by the National Natural Science Foundation of China(Grant No. 11971337).}
}
\date{}

\begin{document}
\maketitle

\parskip 10pt
\thispagestyle{empty}

\begin{abstract}
\indent In this paper, we propose an efficient proper orthogonal decomposition based reduced-order model(POD-ROM) for nonstationary Stokes equations, which combines the classical projection method with POD technique. This new scheme mainly owns two advantages: the first one is low computational costs since the classical projection method decouples the reduced-order velocity variable and reduced-order pressure variable, and POD technique further improves the computational efficiency; the second advantage consists of circumventing the verification of classical LBB/inf-sup condition for mixed POD spaces with the help of pressure stabilized Petrov-Galerkin(PSPG)-type projection method, where the pressure stabilization term is inherent which allows the use of non inf-sup stable elements without adding extra stabilization terms. We first obtain the convergence of PSPG-type finite element projection scheme, and then analyze the proposed projection POD-ROM's stability and convergence. Numerical experiments validate out theoretical results.
\end{abstract}

\section{Introduction}
\noindent Reduced Order Models(ROMs) have become widespread in the scientific community and owns constant attention in academic community because of its strength and efficiency in alleviating the huge computational costs needed in many modern engineering applications and academic purposes. As one of the most popular ROM approaches, the proper orthogonal decomposition(POD) strategy aims to use singular value decomposition(SVD) to extract the dominant modes(in the sense of energy) from a high-resolution numerical scheme(commonly referred to full order models(FOMs)) to form POD bases/modes. Onto these reduced-order bases, a Galerkin projection can be used to project into the space spanned by these bases/modes to obtain the reduced-order numerical solution, this pioneering work is done by Kunisch and Volkwein in \cite{Galerkin-POD-NM-2001,Galerkin-POD-SINUM-2002} which has laid the foundation of numerical analysis for Galerkin-POD. Other references about Galerkin-POD can be found in \cite{Galerkin-POD-ErrorAnalysis-SINUM-2014, Luo-POD-NS, Luo-POD,Galerkin-POD-Volkwein-LectureNotes-2011} and so on. Applications of Galerkin-POD on incompressible flows, especially with the spatial variables discretized by finite element method(FE-POD), has a long history, from Luo et al.\cite{Luo-POD-NS} obtaining FE-POD-ROM for nonstationary Navier-Stokes equations based on finite difference method and mixed finite element method to discretize temporal and spatial variables, Ravindran\cite{Ravindran-Boussinesq-NMPDE-2010} analyzing FE-POD-ROM for nonstationary Boussinesq equations, to other finite element methods, like discontinuous Galerkin and hybrid discontinuous Galerkin, for spatial discrete are applied to get FE-POD-ROM\cite{POD-HDG}. \\
\indent Traditional mixed finite element POD-ROM had encountered a problem, that is, when considering the classical mixed finite element scheme as the snapshot source of the POD, the pressure variable is eliminated because of the classical LBB/inf-sup condition, so only the velocity finite element solution is used as the snapshot of the POD technique to construct the POD bases/modes, which makes the final obtained FE-POD-ROM does not contain the pressure variable. This makes it improper to use POD technique directly to reduce the model order in some engineering and academic models that need the pressure variables. Among that problem, the pressure field must be reconstructed a posterior using some pressure recovery techniques, like pressure Poisson equation\cite{POD-PressurePoissonEquation-JFM-2005}, Supremizer stabilized technique\cite{POD-Supremizer-2015-IJNME} and so on. However, just as stated in \cite{POD-AC-SINUM-2020} and \cite{NS-POD-2014-JCP}, there are some deficiencies in these a posterior pressure recovery techniques, such as the usage of those techniques still need the fulfillment of LBB/inf-sup condition which restricts some flexible pairs of mixed finite element spaces, and the un-physical and unclear boundary condition of pressure Poisson equation poses a challenge for the utilization of this technique. With those in mind, recently, a new research direction related to FE-POD has aroused some researchers' interest, we refer as stabilized FE-POD, which draws inspiration from the successful applications of stabilized techniques in finite element method. Caiazzo et al.\cite{NS-POD-2014-JCP} firstly proposed to combine stabilization technique, like residual-based streamline-upwind Petrov-Galerkin(SUPG) stabilization, with FE-POD to preserve the finite element pressure solution, so that the POD reduced-order pressure can be obtained from the snapshot of finite element pressure solution. From then on, other stabilized techniques were introduced into FE-POD, like local projection stabilization(LPS) FE-POD\cite{POD-LPS-SINUM-2020,POD-LPS-SINUM-2021}, artificial compressible(AC) FE-POD\cite{POD-AC-SINUM-2020}, streamline diffusion(SD) FE-POD\cite{POD-SD-2021-JCP} and so on. \\
\indent In this paper, our aim is to combine the classical projection method and POD technique to propose an efficient projection POD-ROM. As one of the fast decoupled numerical algorithms for solving incompressible fluid, projection method has become a classical one after it was proposed by Chorin and Temam \cite{Proj-Chorin-1969-MC, Proj-Teman-1968-French} and developed by Shen and Guermond \cite{Proj-Shen-1992-NM,Proj-Guermond-1998-NM,Proj-Shen-1992-SINUM,Proj-Overview-2006-CMAME} etc. Apart from that, as mentioned by Rannacher in \cite{Rannacher-1992-Projection}, another remarkable feature of this method is that the original Chorin-Temam projection method owns an inherent mechanism of pressure stability. This mechanism was proposed by Rannacher \cite{ Rannacher-1992-Projection} and was realized by Frutos et al. \cite{Frutos-ModifiedProjection-2018-AMC}. These two advantages of the classical projection method are merged into the POD technique to get our projection POD-ROM. The main difference in our approach with respect to those stabilized FE-POD-ROM is that the pressure stabilized Petrov-Galerkin(PSPG)-like FE-POD-ROM has the inherent pressure stabilization, which means no additional pressure stabilization term is needed to circumvent the LBB/inf-sup condition of mixed POD spaces which is still an open problem to verify. \\
\indent The main contribution of this paper is that a more efficient ROM numerical scheme for solving incompressible fluid is obtained by combining the advantageous computational efficiency of both projection method and POD technique. This POD-ROM scheme stems from replacing the finite element functions in the fully discrete scheme of the projection method with the POD reduced-order functions. To this end, we analyze the fully discrete scheme of classical projection method in the form of with inherent pressure stabilized term, and numerically verify some theoretical results concerning with this finite element scheme. Furthermore, we take the finite element solution as snapshots in order to use method of snapshots to construct POD modes and POD spaces, and obtain our projection POD scheme. The stability and convergence of this scheme are analyzed, and which are also validated by some numerical experiments.          \\
\indent The outline of this paper is as follows: In section 2, we introduce some essential notations. In section 3, the classical projection method will be presented, followed by its stability and convergence analysis. In section 4, we propose the projection POD numerical scheme and prove its stability and convergence. In section 5, some numerical experiments will be investigated to confirm the theoritical results analyzed before. Finally, in section 6, we draw conclusions to complete this paper.
\section{Preliminaries and notations}
\noindent we consider the nonstationary Stokes equations
\begin{equation}\label{Evo-Stokes}
\left\{
\begin{aligned}
\partial_t \boldsymbol{u} - \nu\Delta \boldsymbol{u} + \nabla p &= \boldsymbol{f},\quad \text{in}\;\Omega,\\
\nabla\cdot \boldsymbol{u} &= 0,\quad \text{in}\;\Omega,\\
\boldsymbol{u} &= \boldsymbol{0},\quad \text{on}\;\partial\Omega,\\
\boldsymbol{u}(0,x) &= \boldsymbol{u}_0(x),\quad \text{in}\;\partial\Omega,
\end{aligned}
\right.
\end{equation}
where $\Omega$ is an open bounded domain in $\mathbb{R}^2$ with a sufficiently smooth boundary $\partial\Omega$. The unknowns are the vector function $\boldsymbol{u}$(velocity) and the scalar function $p$(pressure).\\
\indent In order to state the above problem properly, we need some settings. Let $L^p(\Omega),\;1\leq p\leq\infty$ denote the space of standard p-th power absolutely integrable functions with respect to the Lebesgue measure. In particular, $L^2(\Omega)$ is a Hilbert space endowed with the scalar product $(\cdot,\cdot)$ and its induced norm $\|\cdot\|_0$. Moreover, we use $\|\cdot\|^2_{\ell^2(L^2)}:= \Delta t\sum^{N}_{n=1}\|\cdot\|^2_0$ to denote the discrete integral in time with respect to $\|\cdot\|^2_0$. The Sobolev space $W^{m,p}(\Omega)$ and its norm $\|\cdot\|_{m,p}$ are also standard in the sense of \cite{Ciarlet-FEM-2002}. Furthermore, we use the abbreviation $H^m(\Omega):=W^{m,2}(\Omega)$ and $\|\cdot\|_m:=\|\cdot\|_{m,2}$, it's a Hilbert space with a scalar product. Moreover, we denote by $H^1_0(\Omega)$ the space of functions of $H^1(\Omega)$ with vanishing trace on $\partial\Omega$ and by $H^{-1}(\Omega)$ its dual space. The symbol $\langle\cdot,\cdot\rangle$ in general denotes the duality pairing between the space and it's dual one. Vector analogues of the Sobolev spaces along with vector valued functions are denoted by boldface letters, for instance $\boldsymbol{H}^m(\Omega):=\left(H^m(\Omega)\right)^2$. We all know that by Poincar\'{e} inequality\cite[p.135]{Brenner2008}, the seminorm $|u|_1=\|\nabla u\|_0$ in $H^1(\Omega)$ is a norm in $H^1_0(\Omega)$.
\\
\indent Throughout this paper, we use $C$ to denote a positive constant independent of $\Delta t$, $h$, not necessarily the same at each occurrence.\\
\indent Let us introduce some convenient spaces. The first one is
$$
\boldsymbol{H}({\rm div};\Omega):=\{\boldsymbol{u}\in\boldsymbol{L}^2(\Omega)|\; \nabla\cdot\boldsymbol{u}\in\boldsymbol{L}^2({\Omega})\},
$$
which is a Hilbert space with norm $\|\boldsymbol{u}\|_{{\rm div}}:=\|\boldsymbol{u}\|_0+\|\nabla\cdot\boldsymbol{u}\|_0$, and
$$
\boldsymbol{H}_0({\rm div};\Omega):= \{\boldsymbol{u}\in \boldsymbol{H}({\rm div};\Omega)|\; \boldsymbol{u}\cdot\boldsymbol{n}|_{\partial\Omega}=0\}.
$$
\indent For future use, we also need the following spaces, which play a key role in Helmholtz decomposition,
$$
\boldsymbol{J}_0 := \left\{\boldsymbol{v}\in \boldsymbol{H}^1_0(\Omega):\; \nabla\cdot \boldsymbol{v} = 0 \right\},
$$
and
$$
\boldsymbol{J}_1 := \left\{\boldsymbol{v}\in \boldsymbol{L}^2(\Omega):\; \nabla\cdot \boldsymbol{v} = 0,\;\text{and}\;\boldsymbol{v}\cdot \boldsymbol{n}|_{\partial\Omega} = 0 \right\}.
$$
\indent By means of spaces above, we give the classical Helmholtz decomposition\cite[p.29]{749} as
\begin{equation}\label{HelmDecom}
\boldsymbol{L}^2(\Omega) = \boldsymbol{J}_1(\Omega) \oplus [\boldsymbol{J}_1(\Omega)]^{\perp} = \boldsymbol{J}_1(\Omega) \oplus \{\nabla q: q\in H^1(\Omega)\}.
\end{equation}
\indent In order to give a variational formulation of problem (\ref{Evo-Stokes}), we consider the velocity space $\boldsymbol{V}:=\boldsymbol{H}^1_0$ and the pressure space $Q:=L^2_0=\{q\in L^2(\Omega)|\int_{\Omega}q\,{\rm d}x=0\}$. For a given time constant $T$, a weak formulation of the nonstationary Stokes equations are expressed as follows: for almost $\forall t\in(0,T]$, find $(\boldsymbol{u},p):(0,T]\rightarrow \boldsymbol{V}\times Q$, such that
\begin{equation}\label{ContVari}
\left\{
\begin{aligned}
(\partial_t \boldsymbol{u},\boldsymbol{v}) + \nu(\nabla \boldsymbol{u},\nabla\boldsymbol{v}) + (\nabla p,\boldsymbol{v}) &= \langle\boldsymbol{f},\boldsymbol{v}\rangle,\quad \forall \boldsymbol{v}\in \boldsymbol{V},\\
(\nabla\cdot \boldsymbol{u},q) &= 0,\quad \forall q\in Q. \\
\boldsymbol{u}(x,0) &= \boldsymbol{u}^0(x).
\end{aligned}
\right.
\end{equation}
\indent Let $\{\mathcal{T}_h\}$ be a uniformly regular family of triangulation of $\overline{\Omega}$(see \cite[p.111]{Ciarlet-FEM-2002}) and $h:=\max_{K\in\mathcal{T}_h}\{h_K|\; h_K\\ :={\rm diam}(K)\}$. Since we want to construct the POD space where snapshots are deriving from both the finite element approximated velocity and also pressure, contrary to the general choice which is to choose the inf-sup stable mixed finite element spaces, we choose equal-order mixed finite element spaces and then it's obviously not stable in the sense of the classical inf-sup condition, as
$$
Y^l_h:=\{v_h\in C^0(\overline{\Omega})|\; v_h|_K\in \mathbb{P}_l(K),\;\forall K\in \mathcal{T}_h\},\; l\geq1,
$$
where $\mathbb{P}_l(K)$ is the space of polynomials up to degree $l$ on $K$. Since there are two types of velocities in the temporal semi-discrete projection scheme, i.e., intermediate velocity $\widetilde{\boldsymbol{u}}^{n+1}$ and end-of-step velocity $\boldsymbol{u}^{n+1}$, we need the three finite element spaces to locate three full discrete variables of three unknowns. To this end, we define the equal-order finite element spaces for end-of-step velocity, intermediate velocity and pressure
$$
\boldsymbol{V}_h:= \boldsymbol{Y}^l_h \cap \boldsymbol{V},\quad \boldsymbol{W}_h:= \boldsymbol{Y}^l_h \cap \boldsymbol{J}_1,\quad Q_h:= Y^l_h \cap Q.
$$
\indent Since the triangulations are assumed to be shape regular, the following inverse inequality holds for each $\boldsymbol{v}_h\in \boldsymbol{V}_h$ on each mesh cell $K\in \{\mathcal{T}_h\}$(see \cite[Theorem 4.5.11]{Brenner2008}),
\begin{equation}\label{Inve-Ineq}
\|\boldsymbol{v}_h\|_{m,p,K} \leq C_{{\rm inv}} h^{l-m+2(\frac{1}{p} - \frac{1}{q})}_{K}\|\boldsymbol{v}_h\|_{l,q,K},
\end{equation}
where $0\leq l\leq m\leq 1$, $0\leq q\leq p\leq \infty$ and $h_k$ is the size of the mesh cell $K\in \{\mathcal{T}_h\}$. Then, $I_h$(resp. $J_h$) denotes a bounded linear interpolation operator $I_h:\; \boldsymbol{W}^{t,q}(K)\rightarrow \boldsymbol{V}_h$(resp. $J_h:\; W^{t,q}(K)\rightarrow Q_h$), where $1\leq q\leq\infty$, $0\leq s\leq t \leq l+1$, that satisfies the following optimal bounds(see \cite[Theorem 4.4.4]{Brenner2008})
\begin{equation}\label{Inte-Opti-Boun}
\begin{aligned}
\|\boldsymbol{v} - I_h\boldsymbol{v}\|_{s,q,K} &\leq C h_{K}^{t-s}\|\boldsymbol{v}\|_{t,q,K}, \\
\|q - J_hq\|_{s,q,K} &\leq C h_{K}^{t-s}\|q\|_{t,q,K},
\end{aligned}
\end{equation}
and also interpolation operator $J_h$ has following stability
\begin{equation}\label{H1Stab-Jh}
\|\nabla J_h q\|_0 \leq C \|\nabla q\|_0.
\end{equation}
\indent The Sobolev inclusion $L^2(\Omega)\hookrightarrow H^{-1}(\Omega)$ relation implies the following inequality holds:
\begin{equation}\label{SobolevImbed}
\|v\|_{-1} \leq C \|v\|_0,\quad \forall v\in L^2(\Omega).
\end{equation}

\section{Classical projection method}

\subsection{classical projection scheme}
\noindent The original projection method consists of firstly finding an intermediate velocity $\widetilde{\boldsymbol{u}}^{n+1}$ from the momentum equation without the pressure term, and then utilizing the classical Helmholtz decomposition (\ref{HelmDecom}) to get an end-of-step velocity $\boldsymbol{u}^{n+1}$ which is solenoidal and its normal component $\boldsymbol{u}^{n+1}\cdot \boldsymbol{n}$ vanishes. Meanwhile, as a ``by-product" of the above decomposition, a discrete pressure $p^{n+1}$ can be also get for which some investigations have been made to confirm that this ``by-product" pressure is indeed the proper approximation of the exact pressure $p(t_{n+1})$, see \cite{Rannacher-1992-Projection}.\\
\indent For a given timestep $\Delta t$, we consider a uniform discretization of the time interval $[0,T]$ as $0=t_0 < t_1 < \cdots < t_N$ where $N=\lfloor T/\Delta t \rfloor$ is the rounding $T/\Delta t$ down. In mathematics, the above procedure can be interpreted as the following classical Chorin-Teman projection method\cite{Proj-Chorin-1968-MC,Proj-Teman-1968-French}: \\
\indent(\textbf{Scheme 1 - Semi-discrete of classical projection method}) For $n=0,1,\ldots,N-1$, \\
\noindent\textbf{Step 1}- Given $\boldsymbol{u}^n \in \boldsymbol{J}_0$, find $\widetilde{\boldsymbol{u}}^{n+1}\in \boldsymbol{V}$ that satisfies
$$
\frac{\widetilde{\boldsymbol{u}}^{n+1} - \boldsymbol{u}^n}{\Delta t} - \nu\Delta\widetilde{\boldsymbol{u}}^{n+1} = \boldsymbol{f}^{n+1},\quad\text{in}\;\Omega.
$$
\textbf{Step 2}- Given $\widetilde{\boldsymbol{u}}^{n+1}\in \boldsymbol{V}$, compute $(\boldsymbol{u}^{n+1},p^{n+1})\in \boldsymbol{H}_0({\rm div},\Omega)\times \left[Q\cap H^1(\Omega)\right]$ from
$$
\left\{
\begin{aligned}
\boldsymbol{u}^{n+1} + \Delta t\nabla p^{n+1} &= \widetilde{\boldsymbol{u}}^{n+1},\quad\text{in}\;\Omega,\\
\nabla\cdot \boldsymbol{u}^{n+1} &= 0,\quad\text{in}\;\Omega, \\
\boldsymbol{u}^{n+1}\cdot\boldsymbol{n} &= 0,\quad\text{on}\;\partial\Omega.
\end{aligned}
\right.
$$
\indent The usage of classical Helmholtz decomposition is reflected in decomposing the intermediate velocity $\widetilde{\boldsymbol{u}}^{n+1}$ from \textbf{step\ 1} to get a solenoidal velocity $\boldsymbol{u}^{n+1}$ and a gradient of a $H^{1}$-function, which can be stated as:\\
\noindent$\mathbf{Step\ 2^{\prime}}$- Projecting $\widetilde{\boldsymbol{u}}^{n+1}$ onto the space $\boldsymbol{J}_0$
$$
\boldsymbol{u}^{n+1} = P_{\boldsymbol{J}_0}(\widetilde{\boldsymbol{u}}^{n+1}).
$$
\begin{Remark}
Let us explain the relation between $\boldsymbol{Step 2}$ and $\boldsymbol{Step 2^{\prime}}$. Followed by (\ref{HelmDecom}), decomposing $\widetilde{\boldsymbol{u}}^{n+1}$ would directly gets:
\begin{equation}\label{HelmDecom-u}
\widetilde{\boldsymbol{u}}^{n+1} = \boldsymbol{u}^{n+1} + \nabla\phi^{n+1},
\end{equation}
where $\boldsymbol{u}^{n+1}\in \boldsymbol{J}_1$ implies $\nabla\cdot\boldsymbol{u}^{n+1}=0$ and $\boldsymbol{u}^{n+1}\cdot \boldsymbol{n}|_{\partial\Omega}=0$, which is exactly the second and third equations in $\boldsymbol{Step\ 2}$. Multiplying the previous decomposition by $1/\Delta t$ and adding it to the first equation in $\boldsymbol{Step\ 1}$, we get $(\boldsymbol{u}^{n+1}-\boldsymbol{u}^{n})/\Delta t-\nu\Delta \widetilde{\boldsymbol{u}}^{n+1} + \nabla(\phi^{n+1}/\Delta t)=\boldsymbol{f}^{n+1}$, which means $\phi^{n+1}/\Delta t$ can be regarded as a proper approximation of $p(t_{n+1})$, i.e., we can define $p^{n+1}:=\phi^{n+1}/\Delta t$. Thus, on the one hand, we can compute $\phi^{n+1}$ in the former decomposition (\ref{HelmDecom-u}) and using $p^{n+1}:=\phi^{n+1}/\Delta t$ to get $p^{n+1}$; on the other hand, we can also substitute $\phi^{n+1}$ with $\Delta tp^{n+1}$ in (\ref{HelmDecom-u}) to directly compute $p^{n+1}$, which is exactly what the first equation in \textbf{step 2} does.
\end{Remark}
Or, Step 2 is equivalent to the following pressure Poisson equation using $\nabla\cdot \boldsymbol{u}^{n+1} = 0$:\\
\noindent$\mathbf{Step\ 2^{\prime\prime}}$- Given $\widetilde{\boldsymbol{u}}^{n+1}\in \boldsymbol{V}$, compute $p^{n+1}\in Q\cap H^2(\Omega)$ from
$$
\left\{
\begin{aligned}
\Delta p^{n+1} &= \frac{1}{\Delta t}\nabla\cdot \widetilde{\boldsymbol{u}}^{n+1},\quad\text{in}\;\Omega,\\
\frac{\partial p^{n+1}}{\partial \boldsymbol{n}} &= 0,\quad\text{on}\;\partial\Omega.
\end{aligned}
\right.
$$
\indent Based on the previous finite element spaces, we derive the fully discrete scheme of projection method.\\
\indent(\textbf{Scheme 2 - Fully-discrete of classical projection scheme}) Let $\boldsymbol{u}^0_h$ be the Lagrangian interpolation or Ritz projection onto $\boldsymbol{V}_h$ of $\boldsymbol{u}_0$, then for $n=0,1,\ldots,N-1$, we can compute $(\widetilde{\boldsymbol{u}}^{n+1}_h, \boldsymbol{u}^{n+1}_h, p^{n+1}_h)\in$ $ \boldsymbol{V}_h\times \boldsymbol{W}_h\times Q_h$ iteratively by
\begin{equation}\label{Pro-Method-1}
\left\{
\begin{aligned}
\left(\frac{\widetilde{\boldsymbol{u}}^{n+1}_h - \boldsymbol{u}^n_h}{\Delta t},\boldsymbol{v}_h\right) + \nu(\nabla\widetilde{\boldsymbol{u}}^{n+1}_h, \nabla \boldsymbol{v}_h) &= \langle\boldsymbol{f}^{n+1},\boldsymbol{v}_h\rangle,\quad \forall \boldsymbol{v}_h\in \boldsymbol{V}_h,\\
(\nabla\cdot \widetilde{\boldsymbol{u}}^{n+1}_h,q_h) + \Delta t(\nabla p^{n+1}_h,\nabla q_h) &= 0,\quad \forall q_h\in Q_h,\\
\boldsymbol{u}^{n+1}_h &= \widetilde{\boldsymbol{u}}^{n+1}_h - \Delta t\cdot\nabla p^{n+1}_h.
\end{aligned}
\right.
\end{equation}
\indent As pointed in \cite{Rannacher-1992-Projection}, it is not necessary to compute the projected velocities $\{\boldsymbol{u}^n_h\}^N_{n=0}$ since these quantities can be eliminated. To see this, replacing $\boldsymbol{u}^n_h$ in the first equation with the third equation at $t_n$ in \textbf{Scheme 3}, we reach the following Proj-PSPG-FE scheme.\\
\indent(\textbf{Scheme 3 - Proj-PSPG-FE})\label{Scheme4}
Let $\boldsymbol{u}^0_h$ be the Lagrangian interpolation or Ritz projection onto $\boldsymbol{V}_h$ of $\boldsymbol{u}_0$ and $p^0_h$ denote some proper approximation of initial pressure $p(t_0)$, we set $\widetilde{\boldsymbol{u}}^{0}_h = \boldsymbol{u}^{0}_h + \Delta t\nabla p^{0}_h$ and replace $\boldsymbol{u}^n_h$ with the third equality in \textbf{Scheme 2}, we obtain the following equivalent scheme: for $n=0,1,\ldots,N-1$, find $(\widetilde{\boldsymbol{u}}^{n+1}_h,p^{n+1}_h)\in \boldsymbol{V}_h\times Q_h$
\begin{equation}\label{Pro-Method-2}
\left\{
\begin{aligned}
\left(\frac{\widetilde{\boldsymbol{u}}^{n+1}_h - \widetilde{\boldsymbol{u}}^n_h}{\Delta t},\boldsymbol{v}_h\right) + \nu(\nabla\widetilde{\boldsymbol{u}}^{n+1}_h, \nabla \boldsymbol{v}_h) + (\nabla p^n_h, \boldsymbol{v}_h) &= \langle\boldsymbol{f}^{n+1},\boldsymbol{v}_h\rangle,\quad \forall \boldsymbol{v}_h\in \boldsymbol{V}_h,\\
(\nabla\cdot \widetilde{\boldsymbol{u}}^{n+1}_h,q_h) + \Delta t(\nabla p^{n+1}_h,\nabla q_h) &= 0,\quad \forall q_h\in Q_h.
\end{aligned}
\right.
\end{equation}

\begin{Remark}\label{RemarkAbandonFirtFewPress}
For the initial discrete pressure $p^0_h$, we simply take $p^0_h=0$ in later numerical experiments, the reasons behind this are twofold: firstly, it can simplify the solving of initial discrete pressure $p^0_h=0$, compared to other techniques, like pressure Poisson equation used in \cite{Heywood-Rannacher-1982-SINUM-1,John-PSPGAnalysis-SIAM-2015}; secondly, it can also directly get the initial intermediate discrete velocity $\widetilde{\boldsymbol{u}}^{0}_h$ since at this point $\widetilde{\boldsymbol{u}}^{0}_h = \boldsymbol{u}^{0}_h$. We may also observe that this rude treatment would make the first few time steps discrete pressure appear oscillation. Indeed, numerical experiments have showed that the first few time steps discrete pressure values have to be abandoned because of big discrete errors; nevertheless, we will see in later numerical experiments that, only after a few time steps(for example, from $n=6$ on), the discrete pressure values can be adopted as the appropriate discrete approximation of the exact pressure.
\end{Remark}
\begin{Remark}
We note that although the obtained finite element velocity solution in scheme 3 is $\widetilde{\boldsymbol{u}}^{n+1}_h$, rather than $\boldsymbol{u}^{n+1}_h$, we will see in latter numerical experiment that when choosing $\boldsymbol{V}_h\times Q_h = P^1$-$P^1$, the solved $\widetilde{\boldsymbol{u}}^{n+1}_h$ can be used as a proper finite element approximation for the exact velocity $\boldsymbol{u}^{n+1}$. If one wants to obtain $\boldsymbol{u}^{n+1}_h$, we can multiply by the last equation in scheme 2 some test function $\boldsymbol{w}_h\in \boldsymbol{W}_h=\boldsymbol{Y}_h^1\cap \boldsymbol{J}_1$ to get
\begin{equation}\label{CompuEndVelo}
(\boldsymbol{u}^{n+1}_h, \boldsymbol{w}_h) = (\widetilde{\boldsymbol{u}}^{n+1}_h, \boldsymbol{w}_h) - \Delta t(\nabla p^{n+1}_h, \boldsymbol{w}_h).
\end{equation}
then, solving above equation can finally obtain $\boldsymbol{u}^{n+1}_h\in \boldsymbol{W}_h=\boldsymbol{Y}_h^1\cap \boldsymbol{J}_1$.
\end{Remark}
We observe that, in the classical projection scheme (\ref{Pro-Method-1}) or (\ref{Pro-Method-2}), if we set $\Delta t=\mathcal{O}(h^2)$, the existence of PSPG-stabilized term $\Delta t(\nabla p^{n+1}_h,\nabla q_h)$ makes the spatial convergence order of the error can not achieve more than second order in $L^2$ norm of the velocity and first order of the error in the $L^2$ norm of the pressure, so $P^1$-$P^1$ for finite element spaces and $(\boldsymbol{u},p) \in \boldsymbol{H}^2\times H^1$ seem to be the best choice concerning about the efficiency of computation and the regularity for the exact solutions. Thus, We set $l=1$ in the definition of $Y^l_h$, thus the finite element spaces for velocity and pressure would be $\boldsymbol{V}_h=Y^1_h\cap \boldsymbol{V}$ and $Q_h=Y^1_h\cap Q$, respectively.\\
\indent In the sequel we assume $\Delta t = \mathcal{O}(h^2)$. Specifically, we assume there exist two positive constants $C_1,C_2$, such that
\begin{equation}\label{Delt-h2}
C_1h^2 \leq \Delta t \leq C_2h^2.
\end{equation}
\indent The following modified inf-sup condition relaxes the classical inf-sup condition and makes many mixed finite element pairs ``stable" in the sense of this modified one, whose detailed proof can be found in \cite[Lemma 3]{PSPG-Burman-2011-CMAME} or \cite[Lemma 2.1]{John-PSPGAnalysis-SIAM-2015}.
\begin{Lemma}\label{Modi-infsup}
(\textbf{modified inf-sup condition}) Assuming (\ref{Delt-h2}) holds, then for $n=0,1,\ldots,N-1$, we have the following pressure stability
\begin{equation}
\beta\|q_h\|_0 \leq Ch\|\nabla q_h\|_0 + \sup_{\boldsymbol{v}_h\in \boldsymbol{V}_h}\frac{(q_h,\nabla\cdot \boldsymbol{v}_h)}{\|\boldsymbol{v}_h\|_1},\quad\forall q_h\in Q_h.
\end{equation}
\end{Lemma}

\subsection{convergence of classical projection scheme}
In this subsection, we will cite some error estimation results of the finite element solution in the classical projection scheme and omit those proofs, since those estimates have been analyzed in \cite{Frutos-ModifiedProjection-2018-AMC} in detail. The necessity of citing is based on the fact that the following POD snapshots are the finite element solutions, which means that when we analyze the discretization error of the POD-based reduced-order solutions, apart from the ROM truncation error, the discretization error also involve spatial and temporal discretization errors.

\begin{Lemma}\label{Lemma-FE-ConvAnalyVeloPress}
(\textbf{error estimates for Proj-PSPG-FE}) For $n=1,\cdots, N$, Let $(\boldsymbol{u}^{n},p^{n})\in \boldsymbol{X}\times Q$ be the solution of continuous variational form (\ref{ContVari}) at discrete time $t=t_n$, $(\widetilde{\boldsymbol{u}}^{n}_h,p^{n}_h)\in \boldsymbol{X}_h\times Q_h$ is the solution obtained with the PSPG-like projection scheme (\ref{Pro-Method-2}), and assuming (\ref{Delt-h2}) holds, i.e. $\Delta t = \mathcal{O}(h^2)$, then we have
\begin{equation}\label{Equa-FE-ConvAnalyVeloPress}
\begin{aligned}
\|\boldsymbol{u}^n - \widetilde{\boldsymbol{u}}^n_h\|_0 + h\|p^n - p^n_h\|_0 + h\sqrt{\Delta t}\|\nabla(p^n - p^n_h)\|_0  &\leq C(h^2 + \Delta t). \\
\|\nabla(\boldsymbol{u} - \boldsymbol{\widetilde{u}}_h)\|_{l^2(L^2)} + \|p-p_h\|_{l^2(L^2)} &\leq C(h+\sqrt{\Delta t}).
\end{aligned}
\end{equation}
\end{Lemma}
In addition to the error estimates about the intermediate velocity $\widetilde{\boldsymbol{u}}^n_h$ analyzed in \cite{Frutos-ModifiedProjection-2018-AMC}, we can also prove the $L^2$ error estimates about the end-of-step velocity $\boldsymbol{u}^n_h$ which is not made in \cite{Frutos-ModifiedProjection-2018-AMC}. In other words, we can obtain
\begin{Lemma}\label{Lemma-FEL2EndVelo}
For $n=1,\cdots, N$, $\boldsymbol{u}^n_h$ is the velocity solution obtained in Scheme 2, and $\boldsymbol{u}^{n}$ is the velocity solution of continuous variational form (\ref{ContVari}) at discrete time $t=t_n$, then
\begin{equation}
\|\boldsymbol{u}^n - \boldsymbol{u}^n_h\|_0 \leq C(h^2 + \Delta t).
\end{equation}
\end{Lemma}
\begin{proof}
According to the convergent results given by Lemma \ref{Lemma-FE-ConvAnalyVeloPress}, we can prove the stability of $\|\nabla p^n_h\|_0$. That is, the first inequality in (\ref{Equa-FE-ConvAnalyVeloPress}) implies
$$
\|\nabla(p^n - p^n_h)\|_0 \leq C(h/\sqrt{\Delta t} + \sqrt{\Delta t}/h) \leq C,
$$
where we have used (\ref{Delt-h2}). Then
$$
\|\nabla p^n_h\|_0 \leq \|\nabla(p^n - p^n_h)\|_0 + \|\nabla p^n\|_0 \leq C,
$$
where the constant $C$ can be chose to stay nearly fixed with the decrease of $h$ and $\Delta t$. So, based on the result, testing $\boldsymbol{w}_h = \widetilde{\boldsymbol{u}}^{n+1}_h - \boldsymbol{u}^{n+1}_h$ in (\ref{CompuEndVelo}) and rearranging to get
$$
\|\widetilde{\boldsymbol{u}}^{n+1}_h - \boldsymbol{u}^{n+1}_h\|_0 \leq \Delta t\|\nabla p^{n+1}_h\|_0 \leq C\Delta t.
$$
Finally, utilizing triangle inequality and Lemma \ref{Lemma-FE-ConvAnalyVeloPress},
$$
\|\boldsymbol{u}^{n+1} - \boldsymbol{u}^{n+1}_h\|_0 \leq \|\boldsymbol{u}^{n+1} - \widetilde{\boldsymbol{u}}^{n+1}_h\|_0 + \|\widetilde{\boldsymbol{u}}^{n+1}_h - \boldsymbol{u}^{n+1}_h\|_0 \leq C(h^2 + \Delta t).
$$
\end{proof}

\section{Projection POD}
\subsection{POD}
\noindent In this subsection, we will briefly give some essential ingredients for constructing POD method, more details for constructing POD space and POD-based reduced-order models can be found in \cite{Luo-POD,Galerkin-POD-Volkwein-LectureNotes-2011}. As we remarked in Remark \ref{RemarkAbandonFirtFewPress}, the numerical pressure oscillation in the first few time steps caused by setting $p^0_h=0$ makes us have to discard the finite element solutions in the previous time steps, so when using finite element solution obtained by solving Scheme 4 to construct POD spaces, we also need to consider this situations. In view of the fact that the different choice of initial pressure values will affect the discrete error of finite element pressure in the previous time steps, we uniformly denote $n=n_0,n_0\geq 1$ as the starting time steps of taking snapshots in Scheme 4; therefore, for a given positive integer $M$, we choose the finite element solution $(\widetilde{\boldsymbol{u}}^i_h,p^i_h),\;n_0\leq i\leq n_0+M-1$ in Scheme 4, and its difference quotients(DQs) $(\partial \widetilde{\boldsymbol{u}}^i_h,\partial p^i_h),\;n_0+1\leq i\leq n_0+M-1$, where the DQs $\partial f^i$ are defined by $\partial f^i := (f^i - f^{i-1})/\Delta t$ for some function $f$ at discrete time $t=t_i$, to formulate the snapshot spaces
$$
\begin{aligned}
\widetilde{\mathcal{U}} &:= \langle \widetilde{\boldsymbol{u}}^{n_0}_h,\widetilde{\boldsymbol{u}}^{n_0+1}_h,\cdots,\widetilde{\boldsymbol{u}}^{M+n_0-1}_h, \partial\widetilde{\boldsymbol{u}}^{n_0+1}_h, \partial\widetilde{\boldsymbol{u}}^{n_0+2}_h, \cdots, \partial\widetilde{\boldsymbol{u}}^{M+n_0-1}_h\rangle, \\
\mathcal{P} &:= \langle p^{n_0}_h,p^{n_0+1}_h,\cdots,p^{M+n_0-1}_h, \partial p^{n_0+1}_h, \partial p^{n_0+2}_h, \cdots, \partial p^{M+n_0+1}_h \rangle.
\end{aligned}
$$
where $\langle S \rangle$ denotes the space spanned by the set $S$ and we denote by $N_s=2M-1$ the number of snapshots. Let $d_{\widetilde{u}},d_p$ be the dimensions of the spaces $\widetilde{\mathcal{U}}$ and $\mathcal{P}$, respectively. The symbols $K_{\widetilde{u}},K_p$ are correlation matrices corresponding to the snapshots $K_{\widetilde{u}}=\left((K^{\widetilde{u}}_{i,j})\right)\in \mathbb{R}^{N_s\times N_s}$ and $K_p=\left((K^p_{i,j})\right)\in \mathbb{R}^{N_s\times N_s}$ where
$$
K^{\widetilde{u}}_{i,j} := \left(\widetilde{\boldsymbol{u}}^i_h,\widetilde{\boldsymbol{u}}^j_h\right), \quad K^p_{i,j} := \left(p^i_h,p^j_h\right).
$$
and $(\cdot,\cdot)$ is the $L^2$ inner product. Just as in \cite{Galerkin-POD-NM-2001}, a singular value decomposition(SVD) is carried out and the leading generalized eigenfunctions are chosen as bases, referred to as the POD bases. We denote by $\lambda_1\geq \lambda_2\geq \cdots \lambda_{d_{\widetilde{u}}}>0$ the positive eigenvalues of $K_{\widetilde{u}}$ and by $\boldsymbol{x}_1,\boldsymbol{x}_2,\ldots,\boldsymbol{x}_{d_{\widetilde{u}}}\in \mathbb{R}^{N_s}$ the associated eigenvectors. Analogously, $\{\gamma_i,\boldsymbol{y}_i\}^{d_p}_{i=1}$ are the eigen-pairs of $K_p$. Then, the two POD bases can be written explicitly as
\begin{equation}\label{POD-bases}
\begin{aligned}
\boldsymbol{\varphi}_i &= \frac{1}{\sqrt{\lambda_i}}\left(\sum^{n_0+M-1}_{j=n_0}x^j_i \widetilde{\boldsymbol{u}}^j_h + \sum^{n_0+M-1}_{j=n_0+1}x^{j+M-1}_i \partial\widetilde{\boldsymbol{u}}^{j}_h\right),\;&\text{for}\;i=1,2,\cdots,d_{\widetilde{u}},\\
\psi_i &= \frac{1}{\sqrt{\gamma_i}}\left(\sum^{n_0+M-1}_{j=n_0}y^j_i p^j_h + \sum^{n_0+M-1}_{j=n_0}y^{j+M-1}_i \partial p^{j}_h\right),\;&\text{for}\;i=1,2,\cdots,d_{p}.
\end{aligned}
\end{equation}
In what follows we will denote by
$$
\begin{aligned}
\boldsymbol{V}_r = \langle \boldsymbol{\varphi}_1,\boldsymbol{\varphi}_2,\cdots,\boldsymbol{\varphi}_r\rangle,\quad Q_r = \langle \psi_1,\psi_2,\cdots,\psi_r\rangle.
\end{aligned}
$$
i.e., we choose the first $r$ POD bases to span the POD spaces, and define the traditional $L^2$ projection into the POD velocity space $\boldsymbol{V}_r$ and POD pressure space $Q_r$ as
\begin{Definition}
Let $\Pi^v_r:L^2(\Omega)\rightarrow \boldsymbol{V}_r$ and $\Pi^q_r:L^2(\Omega)\rightarrow Q_r$ such that
$$
\begin{aligned}
(\boldsymbol{u} - \Pi^v_r\boldsymbol{u}, \boldsymbol{v}_r) &= 0 \quad \forall \boldsymbol{v}_r\in \boldsymbol{V}_r \quad and \\
(p - \Pi^q_rp, q_r) &= 0 \quad \forall q_r\in Q_r.
\end{aligned}
$$
\end{Definition}
Then, we can get the following Proj-POD-ROM scheme: \vspace{0.2cm}\\
\indent(\textbf{Scheme 4 - Proj-POD-ROM}) For some $n_0\geq 1$, let $(\widetilde{\boldsymbol{u}}^{n_0}_r,p^{n_0}_r)=(\Pi^v_r \widetilde{\boldsymbol{u}}^{n_0}_h,\Pi^q_rp^{n_0}_h)=(\sum^r_{i=1}(\widetilde{\boldsymbol{u}}^{n_0}_h,\boldsymbol{\varphi}_i)\boldsymbol{\varphi}_i$, $\sum^r_{i=1}(p^{n_0}_h,\psi_i)\psi_i)$ be the Galerkin projection of $(\widetilde{\boldsymbol{u}}^{n_0}_h,p^{n_0}_h)$ onto POD spaces $\boldsymbol{V}_r\times Q_r$, then for $n=n_0,n_0+1,n_0+2,\cdots,N-1$, we can compute $(\widetilde{\boldsymbol{u}}^{n+1}_r, p^{n+1}_r)\in \boldsymbol{V}_r\times Q_r$ from
\begin{equation}\label{Pro-POD}
\left\{
\begin{aligned}
\left(\frac{\widetilde{\boldsymbol{u}}^{n+1}_r - \widetilde{\boldsymbol{u}}^n_r}{\Delta t},\boldsymbol{v}_r\right) + \nu(\nabla\widetilde{\boldsymbol{u}}^{n+1}_r, \nabla \boldsymbol{v}_r) + (\nabla p^n_r, \boldsymbol{v}_r) &= \langle\boldsymbol{f}^{n+1},\boldsymbol{v}_r\rangle,\quad \forall \boldsymbol{v}_r\in \boldsymbol{V}_r,      \\
(\nabla\cdot \widetilde{\boldsymbol{u}}^{n+1}_r,q_r) + \Delta t(\nabla p^{n+1}_r,\nabla q_r) &= 0,\quad \forall q_r\in Q_r.
\end{aligned}
\right.
\end{equation}
\subsection{stability and convergence analysis of projection-POD scheme}
\indent We will carry on some numerical analysis about the newly proposed Proj-POD-ROM. To this end, we first recall the classical conclusions concerning about the projection error between the snapshots solution and POD-based reduced-order solution(see \cite{Galerkin-POD-NM-2001})
\begin{Lemma}
(\textbf{time-averaged projection error estimates}) We have the following time-averaged projection error estimates.
\begin{equation}
\begin{aligned}
\frac{1}{{N_s}}\sum^{n_0+M-1}_{j=n_0}\left\|\widetilde{\boldsymbol{u}}^j_h - \sum^{r}_{k=1}\left(\widetilde{\boldsymbol{u}}^j_h,\boldsymbol{\varphi}_k\right)\boldsymbol{\varphi}_k\right\|^2_0 + \frac{1}{{N_s}}\sum^{n_0+M-1}_{j=n_0+1}\left\|\partial\widetilde{\boldsymbol{u}}^j_h - \sum^{r}_{k=1}\left(\partial\widetilde{\boldsymbol{u}}^j_h,\boldsymbol{\varphi}_k\right)\boldsymbol{\varphi}_k\right\|^2_0 &= \sum^{d_{\widetilde{u}}}_{k=r+1}\lambda_k, \\
\frac{1}{{N_s}}\sum^{n_0+M-1}_{j=n_0}\left\|p^j_h - \sum^{r}_{k=1}\left(p^j_h,\psi_k\right)\phi_k\right\|^2_0 + \frac{1}{{N_s}}\sum^{n_0+M-1}_{j=n_0+1}\left\|\partial p^j_h - \sum^{r}_{k=1}\left(\partial p^j_h,\psi_k\right)\phi_k\right\|^2_0 &= \sum^{d_p}_{k=r+1}\gamma_k.
\end{aligned}
\end{equation}
\end{Lemma}
\begin{Remark}
We remark that choosing $L^2$ inner product to construct POD bases and by virtue of Parseval's identity makes the above three identities actually represent $L^2$ POD projection errors.
\end{Remark}
\vspace{-0.5cm}
Apart from the above time-averaged projection error estimates, we also need the following up-to-data optimal in time error estimates(see, \cite{Koc-TimePointwise-SINUM-2021}), which is presented in the form of hypothesis in the previous papers and is proved to be valid in the presence of DQs in \cite{Koc-TimePointwise-SINUM-2021}.
\begin{Lemma}\label{Opti-PointwiseTime}
(\textbf{POD Optimal pointwise in time error estimate}) We have the following optimal in time projection errors
$$
\begin{aligned}
\max_{0\leq k \leq N}\|\boldsymbol{\widetilde{u}}^k_h - \Pi^v_r\boldsymbol{\widetilde{u}}^k_h\|^2_0 &\leq C\sum^{d_{\widetilde{u}}}_{i=r+1}\lambda_i, \\
\max_{0\leq k \leq N}\|p^k_h - \Pi^q_rp^k_h\|^2_0 &\leq C\sum^{d_p}_{i=r+1}\gamma_i,
\end{aligned}
$$
where $C=6\max\{1,T^2\}$, and
$$
\begin{aligned}
\max_{0\leq k \leq N}\|\nabla(\boldsymbol{\widetilde{u}}^k_h - \Pi^v_r\boldsymbol{\widetilde{u}}^k_h)\|^2_0 &\leq C\sum^{d_{\widetilde{u}}}_{i=r+1}\lambda_i\|\nabla\boldsymbol{\varphi}_i\|^2_0, \\
\max_{0\leq k \leq N}\|\nabla(p^k_h - \Pi^q_rp^k_h)\|^2_0 &\leq C\sum^{d_p}_{i=r+1}\gamma_i\|\nabla\psi_i\|^2_0,
\end{aligned}
$$
where $C=6\max\{1,T^2\}$.
\end{Lemma}
\vspace{-0.5cm}
We also need in the later analysis the time-averaged optimal projection error estimate about DQs, which is derived in \cite{Galerkin-POD-NM-2001}.
\begin{Lemma}\label{Opti-TimeAver}
(\textbf{POD Optimal time-averaged error estimate about DQs})
For the difference quotients we have the estimate:
$$
\frac{1}{M}\sum^{n_0+M-1}_{k=n_0}\|\partial\boldsymbol{\widetilde{u}}^k_h - \Pi^v_r\partial\boldsymbol{\widetilde{u}}^k_h\|^2_0 \leq C\sum^{d_{\widetilde{u}}}_{i=r+1}\lambda_i.
$$
\end{Lemma}
\vspace{-0.5cm}
The following analogues of finite element inverse inequalities will be used in later analysis(see, \cite{Galerkin-POD-NM-2001})
\begin{Lemma}
(\textbf{POD inverse estimates}) Denote by $S_{\widetilde{u}}=\left((s^{\widetilde{u}}_{i,j})\right)\in \mathbb{R}^{{N_s}\times {N_s}}$ and $S_p=\left((s^p_{i,j})\right)\in \mathbb{R}^{{N_s}\times {N_s}}$ are stiffness matrices for POD spaces $\boldsymbol{V}_r,Q_r$ respectively, with entries $s^{\widetilde{u}}_{i,j}=(\nabla\boldsymbol{\varphi}_i,\nabla\boldsymbol{\varphi}_j),s^p_{i,j}=(\nabla\psi_i,\nabla\psi_j)$, and by $\normmm{\cdot}_2$ the spectral norm of the matrix, then
\begin{equation}
\begin{aligned}
\|\nabla \widetilde{\boldsymbol{v}}\|_0 &\leq \sqrt{\normmm{\boldsymbol{S}_{\widetilde{u}}}_2}\|\widetilde{\boldsymbol{v}}\|_0,\ \text{for}\ \forall \widetilde{\boldsymbol{v}}\in\boldsymbol{V}_r,  \\
\|\nabla q\|_0 &\leq \sqrt{\normmm{S_p}_2}\|q\|_0,\ \text{for}\;\forall q\in Q_r.
\end{aligned}
\end{equation}
\end{Lemma}
\begin{Remark}
Mass matrix $M$, for both $(\boldsymbol{\varphi}_i,\boldsymbol{\varphi}_j)$ and $(\psi_i,\psi_j)$, equals identity $I$ since the standard orthogonality of POD base in the sense of $L^2$ inner product.
\end{Remark}
\vspace{-0.5cm}
We can prove the unconditional stability of the Proj-POD-ROM scheme.
\begin{Theorem}(\textbf{Unconditional stability of Proj-POD-ROM})
We have the following unconditional stability
\begin{equation}\label{Sta-POD-Method}
\begin{aligned}
\left\|\widetilde{\boldsymbol{u}}_{r}^{N}\right\|_{0}^{2}+\sum_{n=n_0}^{N-1}\left\|\widetilde{\boldsymbol{u}}_{r}^{n+1}-\widetilde{\boldsymbol{u}}_{r}^{n}\right\|_{0}^{2} &+ \Delta t \sum_{n=n_0}^{N-1}\left(\nu\left\|\nabla \widetilde{\boldsymbol{u}}_{r}^{n+1}\right\|_{0}^{2}+\Delta t\left\|\nabla p_{r}^{n+1}\right\|_{0}^{2},\right) \\
\leq & C\left(\left\|\widetilde{\boldsymbol{u}}_{r}^{0}\right\|_{0}^{2}+\frac{\Delta t}{\nu} \sum_{n=n_0}^{N-1}\|\boldsymbol{f}^{n+1}\|_{-1}\right).
\end{aligned}
\end{equation}
\end{Theorem}
\begin{proof}
Taking $(\boldsymbol{v}_r,q_r)=(\widetilde{\boldsymbol{u}}^{n+1}_r,p^{n}_r)$ in (\ref{Pro-POD}) and multiply by $2\Delta t$, adding both equations and integrating by parts, adding and subtracting $2\Delta t^2\|\nabla p_{r}^{n+1}\|^2_0$, we get
%
\begin{equation}\label{1.1}
\begin{aligned}
\left\|\widetilde{\boldsymbol{u}}_{r}^{n+1}\right\|_{0}^{2}-\left\|\widetilde{\boldsymbol{u}}_{r}^{n}\right\|_{0}^{2}&+\left\|\widetilde{\boldsymbol{u}}_{r}^{n+1}
-\widetilde{\boldsymbol{u}}_{r}^{n}\right\|_{0}^{2} + 2\nu \Delta t\left\|\nabla \widetilde{\boldsymbol{u}}_{r}^{n+1}\right\|_{0}^{2} + 2\Delta t^2\|\nabla p_{r}^{n+1}\|^2_0, \\
&\leq C\frac{\Delta t}{\nu} \|\boldsymbol{f}^{n+1}\|_{-1} + \nu \Delta t\left\|\nabla \widetilde{\boldsymbol{u}}_{r}^{n+1}\right\|_{0}^{2} + 2\Delta t^2\left(\nabla p_{r}^{n+1}, \nabla (p_{r}^{n+1}-p_{r}^{n})\right).
\end{aligned}
\end{equation}
For the last term above, we utilize the second equation in (\ref{Pro-POD}) twice to obtain
$$
\Delta t\left(\nabla p_{r}^{n+1}, \nabla (p_{r}^{n+1}-p_{r}^{n})\right) = -(\nabla\cdot (\widetilde{\boldsymbol{u}}^{n+1}_r - \widetilde{\boldsymbol{u}}^{n}_r),p_{r}^{n+1}) = (\widetilde{\boldsymbol{u}}^{n+1}_r - \widetilde{\boldsymbol{u}}^{n}_r,\nabla p_{r}^{n+1}).
$$
This equality implies
$$
2\Delta t^2\left(\nabla p_{r}^{n+1}, \nabla (p_{r}^{n+1}-p_{r}^{n})\right) \leq \frac23\|\widetilde{\boldsymbol{u}}^{n+1}_r - \widetilde{\boldsymbol{u}}^{n}_r\|^2_0 + \frac32\Delta t^2\|\nabla p_{r}^{n+1}\|^2_0.
$$
Inserting the above equality into (\ref{1.1}) and adding $n$ from $0$ to $N-1$ we finally obtain (\ref{Sta-POD-Method}).
\end{proof}
We have the following truncation error estimate between the POD reduced-order solution and the snapshot data.
\begin{Lemma}\label{Lemma-projFEsolu-projPODsolu}
(\textbf{POD truncation error}) For $n=1,2,\cdots,N$, $(\widetilde{\boldsymbol{u}}^n_h,p^n_h)$ is the solution of finite element projection scheme (\ref{Pro-Method-2}) and $(\widetilde{\boldsymbol{u}}^n_r,p^n_r)$ is the solution of POD projection scheme (\ref{Pro-POD}), then we have the following error estimate:
\begin{equation}
\begin{aligned}
\|\widetilde{\boldsymbol{u}}^N_h - \widetilde{\boldsymbol{u}}^N_r\|^2_0 &+ \nu\Delta t\sum^{N-1}_{n=n_0}\|\nabla(\widetilde{\boldsymbol{u}}^{n+1}_h - \widetilde{\boldsymbol{u}}^{n+1}_r)\|^2_{0} + \Delta t^2\sum^{N-1}_{n=n_0}\|\nabla(p^{n+1}_h - p^{n+1}_r)\|^2_{0}, \\
&\leq C\left(\sum^{d_{\widetilde{u}}}_{i=r+1}\gamma_i\Big(1 + (\nu + 1)\|\nabla\varphi_i\|^2_0\Big) + (\Delta t + 1)\sum^{d_p}_{i=r+1}\epsilon_i\|\nabla\psi_i\|^2_0 + \Delta t^2\right).
\end{aligned}
\end{equation}
\end{Lemma}
\begin{proof}
Finite element projection scheme (\ref{Pro-Method-2}) subtract POD projection scheme (\ref{Pro-POD}) to get:
\begin{equation}
\left\{
\begin{aligned}
\left(\frac{\widetilde{\boldsymbol{u}}^{n+1}_h - \widetilde{\boldsymbol{u}}^n_h}{\Delta t} - \frac{\widetilde{\boldsymbol{u}}^{n+1}_r - \widetilde{\boldsymbol{u}}^n_r}{\Delta t},\boldsymbol{v}_r\right) + \nu(\nabla(\widetilde{\boldsymbol{u}}^{n+1}_h - \widetilde{\boldsymbol{u}}^{n+1}_r), \nabla \boldsymbol{v}_r) + (\nabla (p^n_h - p^n_r), \boldsymbol{v}_r) &= 0,\\
(\nabla\cdot (\widetilde{\boldsymbol{u}}^{n+1}_h - \widetilde{\boldsymbol{u}}^{n+1}_r),q_r) + \Delta t(\nabla (p^{n+1}_h - p^{n+1}_r),\nabla q_r) &= 0.
\end{aligned}
\right.
\end{equation}
Denoting
$$
\begin{aligned}
\widetilde{\boldsymbol{u}}^{n+1}_h - \widetilde{\boldsymbol{u}}^{n+1}_r = \widetilde{\boldsymbol{u}}^{n+1}_h - \Pi^v_r\widetilde{\boldsymbol{u}}^{n+1}_h + \Pi^v_r\widetilde{\boldsymbol{u}}^{n+1}_h - \widetilde{\boldsymbol{u}}^{n+1}_r &:= \widetilde{\boldsymbol{\eta}}^{n+1}_{u,r} + \widetilde{\boldsymbol{w}}^{n+1}_{u,r}, \\
p^{n+1}_h - p^{n+1}_r = p^{n+1}_h - \Pi^q_rp^{n+1}_h + \Pi^q_rp^{n+1}_h - p^{n+1}_r &:= \eta^{n+1}_{p,r} + w^{n+1}_{p,r}.
\end{aligned}
$$
Testing $(\boldsymbol{v}_r,q_r) = (\widetilde{\boldsymbol{w}}^{n+1}_{u,r},w^{n+1}_{p,r})$, we obtain
$$
\left\{
\begin{aligned}
\left(\frac{\widetilde{\boldsymbol{w}}^{n+1}_{u,r} - \widetilde{\boldsymbol{w}}^{n}_{u,r}}{\Delta t}, \widetilde{\boldsymbol{w}}^{n+1}_{u,r}\right) &+ \nu\|\nabla\widetilde{\boldsymbol{w}}^{n+1}_{u,r}\|^2_0 + (\nabla w^{n+1}_{p,r}, \widetilde{\boldsymbol{w}}^{n+1}_{u,r}) = \left(\frac{\widetilde{\boldsymbol{\eta}}^{n+1}_{u,r} - \widetilde{\boldsymbol{\eta}}^{n}_{u,r}}{\Delta t} ,\widetilde{\boldsymbol{w}}^{n+1}_{u,r}\right) + \nu(\nabla\widetilde{\boldsymbol{\eta}}^{n+1}_{u,r} ,\nabla \widetilde{\boldsymbol{w}}^{n+1}_{u,r}), \\
&\quad+ (\nabla(p^{n+1}_h - p^n_h) ,\widetilde{\boldsymbol{w}}^{n+1}_{u,r})  + (\nabla(p^{n}_r - p^{n+1}_r) ,\widetilde{\boldsymbol{w}}^{n+1}_{u,r}) + (\nabla\eta^{n+1}_{p,r} ,\widetilde{\boldsymbol{w}}^{n+1}_{u,r}),\\
(\nabla\cdot \widetilde{\boldsymbol{w}}^{n+1}_{u,r}, w^{n+1}_{p,r}) &+ \Delta t\|\nabla w^{n+1}_{p,r}\|^2_0 = (\nabla\cdot\widetilde{\boldsymbol{\eta}}^{n+1}_{u,r}, \widetilde{\boldsymbol{w}}^{n+1}_{u,r}) + \Delta t(\nabla\eta^{n+1}_{p,r}, \nabla w^{n+1}_{p,r}).
\end{aligned}
\right.
$$
Adding both equations we have the following error equation
\begin{equation}\label{Equa-Temp-6}
\begin{aligned}
\left(\frac{\widetilde{\boldsymbol{w}}^{n+1}_{u,r} - \widetilde{\boldsymbol{w}}^{n}_{u,r}}{\Delta t}, \widetilde{\boldsymbol{w}}^{n+1}_{u,r}\right) &+ \nu\|\nabla\widetilde{\boldsymbol{w}}^{n+1}_{u,r}\|^2_0 + \Delta t\|\nabla w^{n+1}_{p,r}\|^2_0 = \left(\frac{\widetilde{\boldsymbol{\eta}}^{n+1}_{u,r} - \widetilde{\boldsymbol{\eta}}^{n}_{u,r}}{\Delta t} ,\widetilde{\boldsymbol{w}}^{n+1}_{u,r}\right),  \\
& + \nu(\nabla\widetilde{\boldsymbol{\eta}}^{n+1}_{u,r} ,\nabla \widetilde{\boldsymbol{w}}^{n+1}_{u,r}) + (\nabla(p^{n+1}_h - p^n_h) ,\widetilde{\boldsymbol{w}}^{n+1}_{u,r}) + (\nabla(p^{n}_r - p^{n+1}_r) ,\widetilde{\boldsymbol{w}}^{n+1}_{u,r}), \\
&  + (\nabla\eta^{n+1}_{p,r} ,\widetilde{\boldsymbol{w}}^{n+1}_{u,r}) + (\nabla\cdot\widetilde{\boldsymbol{\eta}}^{n+1}_{u,r}, \widetilde{\boldsymbol{w}}^{n+1}_{u,r}) + \Delta t(\nabla\eta^{n+1}_{p,r}, \nabla w^{n+1}_{p,r}), \\
& := A_1 + A_2 + A_3 + A_4 + A_5 + A_6 + A_7.
\end{aligned}
\end{equation}
We will bound the seven terms above separately.
$$
\begin{aligned}
|A_1| &= \left|\left(\frac{\widetilde{\boldsymbol{\eta}}^{n+1}_{u,r} - \widetilde{\boldsymbol{\eta}}^{n}_{u,r}}{\Delta t},\widetilde{\boldsymbol{w}}^{n+1}_{u,r}\right) \right| \leq C\left\|\frac{\widetilde{\boldsymbol{\eta}}^{n+1}_{u,r} - \widetilde{\boldsymbol{\eta}}^{n}_{u,r}}{\Delta t} \right\|^2_0 + \frac{1}{10} \|\widetilde{\boldsymbol{w}}^{n+1}_{u,r}\|^2_0.
\end{aligned}
$$
For the first term above, we utilize the expression of $\widetilde{\boldsymbol{u}}^{n}_h$ and $\Pi^v_r\widetilde{\boldsymbol{u}}^{n}_h$ in the form of POD bases $\boldsymbol{\varphi}_i$ to get
$$
\begin{aligned}
\widetilde{\boldsymbol{\eta}}^{n+1}_{u,r} = \widetilde{\boldsymbol{u}}^{n+1}_h - \Pi^v_r\widetilde{\boldsymbol{u}}^{n+1}_h = \sum^{d_{\widetilde{u}}}_{k=1}\left(\widetilde{\boldsymbol{u}}^{n+1}_h,\boldsymbol{\varphi}_k\right)\boldsymbol{\varphi}_k - \sum^{r}_{k=1}\left(\widetilde{\boldsymbol{u}}^{n+1}_h,\boldsymbol{\varphi}_k\right)\boldsymbol{\varphi}_k = \sum^{d_{\widetilde{u}}}_{k=r+1}\left(\widetilde{\boldsymbol{u}}^{n+1}_h,\boldsymbol{\varphi}_k\right)\boldsymbol{\varphi}_k.
\end{aligned}
$$
Then
$$
\begin{aligned}
\left\|\frac{\widetilde{\boldsymbol{\eta}}^{n+1}_{u,r} - \widetilde{\boldsymbol{\eta}}^{n}_{u,r}}{\Delta t}\right\|^2_0 &= \left\|\frac{1}{\Delta t}\left(\sum^{d_{\widetilde{u}}}_{k=r+1}\left(\widetilde{\boldsymbol{u}}^{n+1}_h,\boldsymbol{\varphi}_k\right)\boldsymbol{\varphi}_k - \sum^{d_{\widetilde{u}}}_{k=r+1}\left(\widetilde{\boldsymbol{u}}^{n}_h,\boldsymbol{\varphi}_k\right)\boldsymbol{\varphi}_k \right)\right\|^2_0, \\
&= \left\|\sum^{d_{\widetilde{u}}}_{k=r+1}\left(\frac{\widetilde{\boldsymbol{u}}^{n}_h - \widetilde{\boldsymbol{u}}^{n}_h}{\Delta t},\boldsymbol{\varphi}_k\right)\boldsymbol{\varphi}_k\right\|^2_0, \\
&= \left\|\sum^{d_{\widetilde{u}}}_{k=r+1}\left(\partial\widetilde{\boldsymbol{u}}^{n+1}_h,\boldsymbol{\varphi}_k\right)\boldsymbol{\varphi}_k\right\|^2_0, \\
&= \left\|\partial\widetilde{\boldsymbol{u}}^{n+1}_h - \Pi^v_r\partial\widetilde{\boldsymbol{u}}^{n+1}_h\right\|^2_0.
\end{aligned}
$$
For other terms on the right-side hand of (\ref{Equa-Temp-6}),
$$
\begin{aligned}
|A_2| &= \left|\nu(\nabla\widetilde{\boldsymbol{\eta}}^{n+1}_{u,r} ,\nabla \widetilde{\boldsymbol{w}}^{n+1}_{u,r}) \right| \leq C\nu\|\nabla\widetilde{\boldsymbol{\eta}}^{n+1}_{u,r}\|^2_0 + \frac12\nu\|\nabla \widetilde{\boldsymbol{w}}^{n+1}_{u,r}\|^2_0 \leq C\nu\sum^{d_{\widetilde{u}}}_{i=r+1}\gamma_i\|\nabla\varphi_i\|^2_0 + \frac12\nu\|\nabla \widetilde{\boldsymbol{w}}^{n+1}_{u,r}\|^2_0,\\
|A_3| &= \left|(\nabla(p^{n+1}_h - p^n_h) ,\widetilde{\boldsymbol{w}}^{n+1}_{u,r}) \right| \leq C\|\nabla(p^{n+1}_h - p^n_h)\|^2_0 + \frac{1}{10}\|\widetilde{\boldsymbol{w}}^{n+1}_{u,r}\|^2_0 \leq C\Delta t\int^{t_{n+1}}_{t_n}\|\nabla(\partial_{t}p_h)\|^2_0\,{\rm d}t + \frac{1}{10}\|\widetilde{\boldsymbol{w}}^{n+1}_{u,r}\|^2_0,\\
|A_4| &= \left| (\nabla(p^{n}_r - p^{n+1}_r) ,\widetilde{\boldsymbol{w}}^{n+1}_{u,r})\right| \leq C\|\nabla(p^{n+1}_r - p^n_r)\|^2_0 + \frac{1}{10}\|\widetilde{\boldsymbol{w}}^{n+1}_{u,r}\|^2_0 \leq C\Delta t\int^{t_{n+1}}_{t_n}\|\nabla(\partial_{t}p_r)\|^2_0\,{\rm d}t + \frac{1}{10}\|\widetilde{\boldsymbol{w}}^{n+1}_{u,r}\|^2_0,\\
|A_5| &= \left|(\nabla\eta^{n+1}_{p,r} ,\widetilde{\boldsymbol{w}}^{n+1}_{u,r}) \right| \leq C\|\nabla\eta^{n+1}_{p,r}\|^2_0 + \frac{1}{10}\|\widetilde{\boldsymbol{w}}^{n+1}_{u,r}\|^2_0 \leq C\sum^{d_p}_{i=r+1}\epsilon_i\|\nabla\psi_i\|^2_0 + \frac{1}{10}\|\widetilde{\boldsymbol{w}}^{n+1}_{u,r}\|^2_0,\\
|A_6| &= \left|(\nabla\cdot\widetilde{\boldsymbol{\eta}}^{n+1}_{u,r}, \widetilde{\boldsymbol{w}}^{n+1}_{u,r}) \right| \leq C|(\nabla\widetilde{\boldsymbol{\eta}}^{n+1}_{u,r}, \widetilde{\boldsymbol{w}}^{n+1}_{u,r})| \leq C\|\nabla\widetilde{\boldsymbol{\eta}}^{n+1}_{u,r}\|^2_0 + \frac{1}{10}\|\widetilde{\boldsymbol{w}}^{n+1}_{u,r}\|^2_0 \leq C\sum^{d_{\widetilde{u}}}_{i=r+1}\gamma_i\|\nabla\varphi\|^2_0 + \frac{1}{10}\|\widetilde{\boldsymbol{w}}^{n+1}_{u,r}\|^2_0,\\
|A_7| &= \left|\Delta t(\nabla\eta^{n+1}_{p,r}, \nabla w^{n+1}_{p,r}) \right| \leq \frac{1}{2}\Delta t\|\nabla\eta^{n+1}_{p,r}\|^2_0 + \frac{1}{2}\Delta t\|\nabla w^{n+1}_{p,r}\|^2_0 \leq C\Delta t\sum^{d_p}_{i=r+1}\epsilon_i\|\nabla\psi_i\|^2_0 + \frac{1}{2}\Delta t\|\nabla w^{n+1}_{p,r}\|^2_0.
\end{aligned}
$$
Combining with all the seven terms' results, we have
$$
\begin{aligned}
\left(\frac{\widetilde{\boldsymbol{w}}^{n+1}_{u,r} - \widetilde{\boldsymbol{w}}^{n}_{u,r}}{\Delta t}, \widetilde{\boldsymbol{w}}^{n+1}_{u,r}\right) &+ \frac12\nu\|\nabla\widetilde{\boldsymbol{w}}^{n+1}_{u,r}\|^2_0 + \frac12\Delta t\|\nabla w^{n+1}_{p,r}\|^2_0  \leq \left\|\partial\widetilde{\boldsymbol{u}}^{n+1}_h - \Pi^v_r\partial\widetilde{\boldsymbol{u}}^{n+1}_h\right\|^2_0,  \\
&\quad + C\Delta t\int^{t_{n+1}}_{t_n}\Big(\|\partial_{t}p_h\|^2_0 + \|\partial_{t}p_r\|^2_0\Big)\,{\rm d}t + C(\nu + 1)\sum^{d_{\widetilde{u}}}_{i=r+1}\gamma_i\|\nabla\varphi_i\|^2_0, \\
&\quad + C(\Delta t + 1)\sum^{d_p}_{i=r+1}\epsilon_i\|\nabla\psi_i\|^2_0 + \frac12\|\widetilde{\boldsymbol{w}}^{n+1}_{u,r}\|^2_0.
\end{aligned}
$$
Multiplying by $\Delta t$, adding with $n$ from $n_0$ to $N-1$ and rearranging, we get
$$
\begin{aligned}
\|\widetilde{\boldsymbol{w}}^{N}_{u,r}\|^2_0 &+ \Delta t\sum^{N-1}_{n=n_0}\|\widetilde{\boldsymbol{w}}^{n+1}_{u,r} - \widetilde{\boldsymbol{w}}^{n}_{u,r}\|^2_0 + \nu\Delta t\sum^{N-1}_{n=n_0}\|\nabla\widetilde{\boldsymbol{w}}^{n+1}_{u,r}\|^2_0 + \Delta t^2\sum^{N-1}_{n=n_0}\|\nabla w^{n+1}_{p,r}\|^2_0, \\
&\quad  \leq \Delta t\sum^{N-1}_{n=n_0}\left\|\partial\widetilde{\boldsymbol{u}}^{n+1}_h - \Pi^v_r\partial\widetilde{\boldsymbol{u}}^{n+1}_h\right\|^2_0 + C(\nu + 1)\sum^{d_{\widetilde{u}}}_{i=r+1}\gamma_i\|\nabla\varphi_i\|^2_0 + \frac12\Delta t\sum^{N-1}_{n=n_0}\|\widetilde{\boldsymbol{w}}^{n+1}_{u,r}\|^2_0, \\
&\quad + C(\Delta t + 1)\sum^{d_p}_{i=r+1}\epsilon_i\|\nabla\psi_i\|^2_0 + C\Delta t^2\int^{T}_{t_{n_0}}\Big(\|\partial_{t}p_h\|^2_0 + \|\partial_{t}p_r\|^2_0\Big)\,{\rm d}t.
\end{aligned}
$$
For the first term on the right-hand side above, we use Lemma \ref{Opti-TimeAver} to have
$$
\begin{aligned}
\Delta t\sum^{N-1}_{n=n_0}\left\|\partial\widetilde{\boldsymbol{u}}^{n+1}_h - \Pi^v_r\partial\widetilde{\boldsymbol{u}}^{n+1}_h\right\|^2_0 \leq C\frac{1}{M}\sum^{M}_{k=1}\|\partial\widetilde{\boldsymbol{u}}^k_h - \Pi^v_r\partial\widetilde{\boldsymbol{u}}^k_h\|^2_0 \leq C\sum^{d_{\widetilde{u}}}_{i=r+1}\gamma_i.
\end{aligned}
$$
By discrete Gronwall inequality and using stability result to bound the last term above, we get
\begin{equation}
\begin{aligned}
\|\widetilde{\boldsymbol{w}}^{N}_{u,r}\|^2_0 &+ \Delta t\sum^{N-1}_{n=n_0}\|\widetilde{\boldsymbol{w}}^{n+1}_{u,r} - \widetilde{\boldsymbol{w}}^{n}_{u,r}\|^2_0 + \nu\Delta t\sum^{N-1}_{n=n_0}\|\nabla\widetilde{\boldsymbol{w}}^{n+1}_{u,r}\|^2_0 + \Delta t^2\sum^{N-1}_{n=n_0}\|\nabla w^{n+1}_{p,r}\|^2_0, \\
&\quad  \leq C\left(\sum^{d_{\widetilde{u}}}_{i=r+1}\gamma_i\Big(1 + (\nu + 1)\|\nabla\varphi_i\|^2_0\Big) + (\Delta t + 1)\sum^{d_p}_{i=r+1}\epsilon_i\|\nabla\psi_i\|^2_0 + \Delta t^2\right).
\end{aligned}
\end{equation}
\end{proof}

Finally, by triangle inequality, Lemma \ref{Lemma-FE-ConvAnalyVeloPress} and Lemma \ref{Lemma-projFEsolu-projPODsolu}, we obtain the following theorem which states the convergence between continuous variational form and POD projection scheme.
\begin{Theorem}\label{POD-ConvAnaly}
(\textbf{error estimate for Proj-POD-ROM}) For $n=n_0,n_0+1,\cdots,N$, let $(\boldsymbol{u}^n,p^n)$ is the solution of (\ref{ContVari}) at $t=t_n$, $(\widetilde{\boldsymbol{u}}^n_r,p^n_r)$ denotes the POD-based reduced-order solutions obtained in (\ref{Pro-POD}), then we have the following convergence estimate:
\begin{equation}
\begin{aligned}
\left\|\boldsymbol{u}^N - \widetilde{\boldsymbol{u}}^N_r\right\|^2_0 &+ \nu\Delta t\sum^{N-1}_{n=n_0}\left\|\nabla(\widetilde{\boldsymbol{u}}^n - \widetilde{\boldsymbol{u}}^n_r)\right\|^2_{L^2} + \Delta t^2\sum^{N-1}_{n=n_0}\left\|\nabla(p - p_r)\right\|^2_{L^2}, \\
&\leq C\left(\sum^{d_{\widetilde{u}}}_{i=r+1}\gamma_i\Big(1 + (\nu + 1)\left\|\nabla\varphi_i\right\|^2_0\Big) + (\Delta t + 1)\sum^{d_p}_{i=r+1}\epsilon_i\left\|\nabla\psi_i\right\|^2_0 + \Delta t^2 + h^2\right).
\end{aligned}
\end{equation}
\end{Theorem}
\section{Numerical Tests}
In this section, we present some numerical experiments to confirm the $a\,priori$ error estimates derived in Theorem \ref{POD-ConvAnaly} for the POD projection scheme 4 - Proj-POD-ROM. To this end, we first determine some necessary parameters used before and verify the convergence in Lemma \ref{Lemma-FE-ConvAnalyVeloPress} and Lemma \ref{Lemma-FEL2EndVelo} of the finite element projection scheme, since with which we formed our POD snapshot spaces, and also the convergence result in Theorem \ref{POD-ConvAnaly} is closely related to the one of the finite element scheme. The open-source finite element package iFEM \cite{iFEM} has been used to run the numerical experiments.

\subsection{Problem setting}
We take $\Omega=[0,1]\times[0,1]\subset \mathbb{R}^2$ and the time interval $(0,1]$ with the viscosity coefficient $\nu=1$ and the prescribed solution
$$
\begin{aligned}
\boldsymbol{u} &=\cos(t)\cdot \left(
                   \begin{array}{c}
                     \pi \sin(\pi x)^2\sin(2\pi y)  \\
                     -\pi \sin(2\pi x)\sin(\pi y)^2 \\
                   \end{array}
                 \right),\\
p &= \cos(t)\cdot 10\cos(\pi x)\cos(\pi y).
\end{aligned}
$$
The right-hand side, the Dirichlet boundary condition and the initial condition are chosen in accordance to the above prescribed solution. \\
\indent We set $\Delta t=0.1h^2$ and all grids are regular $N\times N$ triangular grids with SWNE diagonals for different $N$(i.e., diagonals coming from connecting the southwest and the northeast vertexes on all rectangles), and we take $N=4,8,16,32,64$ sequentially. For simplicity, we use $P^1$-$P^1$ element pair for spatial discretization, and we take the snapshots on the finest computational mesh, i.e., h=1/64, and show this mesh in Figure \ref{FinestMesh}.  \\
\indent When constructing the snapshots spaces, we take into account the fact that the finite element projection scheme involves the initial pressure $p^0_h$ which is not the part of the definition of the problem, so we take $p^0_h=0$ and it might be good to think of the numerical discrete errors of pressure $\|p^n-p^n_h\|_{L^2}$ in the previous steps $n$ are so large that it is not suitable to take finite element solution on those steps into snapshot spaces, thus we choose the finite element solution $(\boldsymbol{u}^n_h,p^n_h)$ and its difference quotients $(\partial\boldsymbol{u}^n_h,\partial p^n_h)$ from $n=n_0=6$ to $n=25$ to formulate snapshots spaces, which means we take $M=20$ and thus the number of snapshots is $N_s=2M-1=39$. \\
\indent The finite element solution $(\boldsymbol{u}^n_{h}, \widetilde{\boldsymbol{u}}^n_j,p^n_h)$ and exact solution $(\boldsymbol{u}^n,p^n)$ at discrete termination time $t=\Delta t\cdot N$ on finest mesh are shown in Figure \ref{VeloPressFigure}, and following POD bases are formed from snapshots via $L^2$ inner product.
\begin{figure}
\centering
\includegraphics[width=8.75cm,height=6.5625cm]{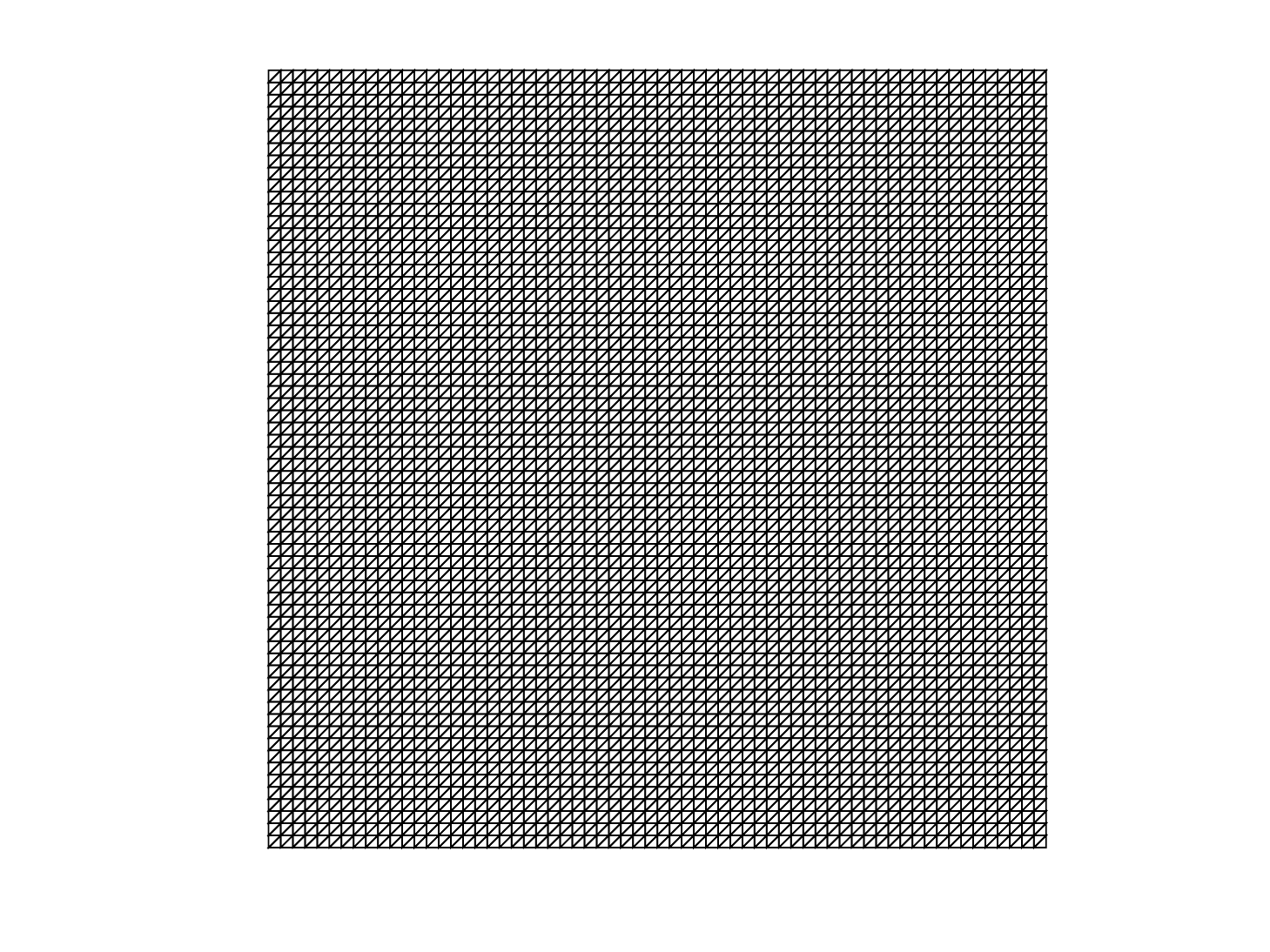}
\caption{Computational mesh with h=1/64.}\label{FinestMesh}
\end{figure}
\begin{figure}
\centering
\includegraphics[width=16.392cm,height=7.812cm]{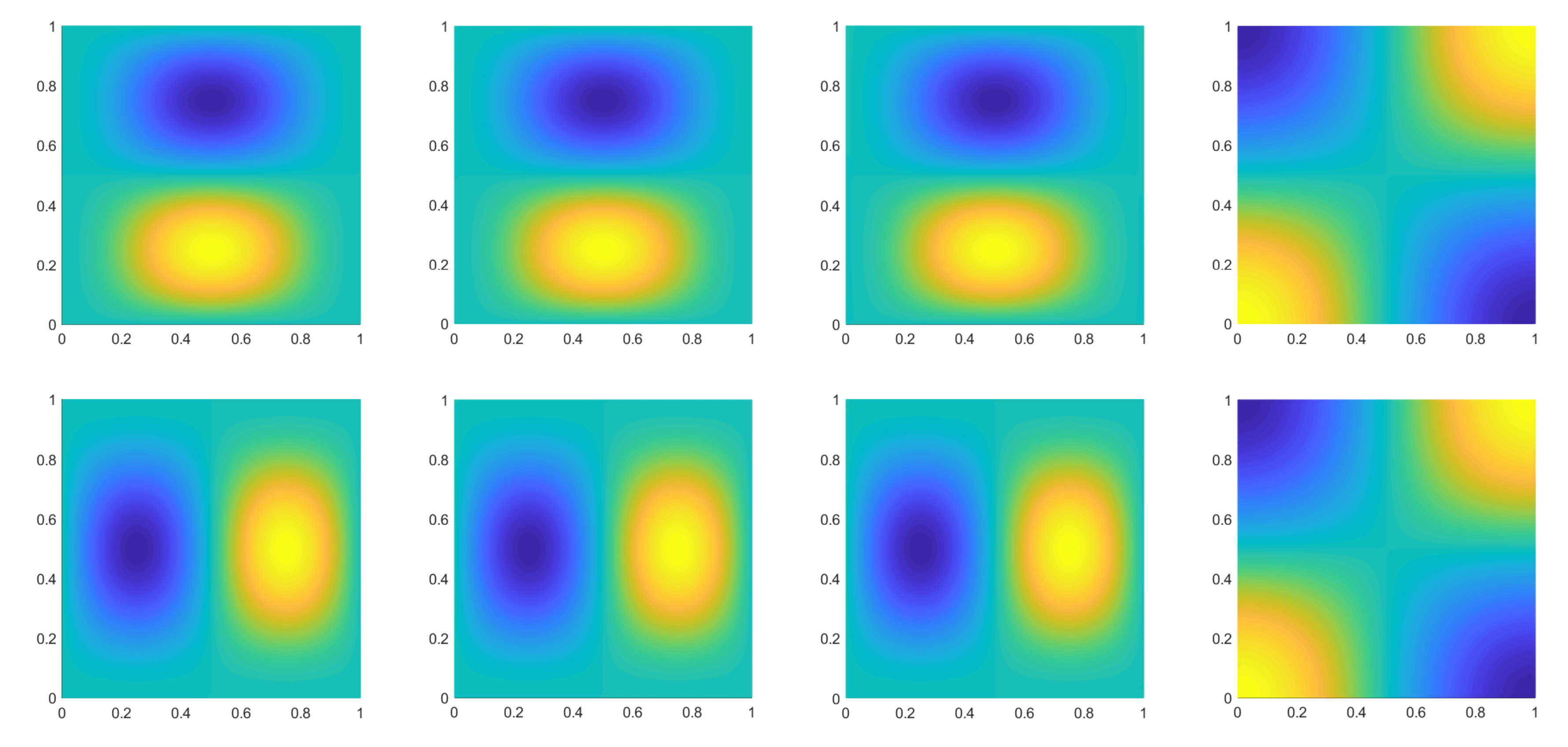}
\caption{The exact solution and finite element solution at $n=N$ on mesh of $h=1/64$. On upper row, from left to right: ($u^n_{1},\widetilde{u}^n_{1,h},u^n_{1,h},p^n$); on lower row, from left to right: ($u^n_{2},\widetilde{u}^n_{2,h},u^n_{2,h},p^n_h$).}\label{VeloPressFigure}
\end{figure}
\subsection{Convergence of the finite element projection scheme}
In this subsection, numerical tests will be used to numerically verify two theoretical results, that is Lemma \ref{Lemma-FE-ConvAnalyVeloPress} and Lemma \ref{Lemma-FEL2EndVelo}, with which we utilize to get convergence analysis Theorem \ref{POD-ConvAnaly} of POD scheme 4. Although Lemma \ref{Lemma-FE-ConvAnalyVeloPress} has been analyzed in \cite{Frutos-ModifiedProjection-2018-AMC}, it lacks the corresponding numerical verifications, so we supplement the numerical tests about Lemma \ref{Lemma-FE-ConvAnalyVeloPress} and also provide that of the Lemma \ref{Lemma-FEL2EndVelo}, which is newly analyzed in this paper. Since $\Delta t=\mathcal{O}(h^2)$, then from the Table \ref{FEVeloConvOrder} and Table \ref{FEPressConvOrder} we can see
$$
\begin{aligned}
\|\nabla(\boldsymbol{u} - \boldsymbol{\widetilde{u}}_h)\|_{l^2(L^2)} + \|p-p_h\|_{l^2(L^2)} &= \mathcal{O}(h), \\
\|\boldsymbol{u}^n - \widetilde{\boldsymbol{u}}^n_h\|_0 + h\|p^n - p^n_h\|_0 + h\sqrt{\Delta t}\|\nabla(p^n - p^n_h)\|_0 &= \mathcal{O}(h^2).
\end{aligned}
$$
which are what the theoretical results in Lemma \ref{Lemma-FE-ConvAnalyVeloPress} shows. Especially, we can see from Table \ref{FEVeloConvOrder}, $\left\|\boldsymbol{u}^n - \boldsymbol{u}^n_h\right\|_{L^2} = \mathcal{O}(h^2)$, which is in accordance of Lemma \ref{Lemma-FEL2EndVelo}.
%
\begin{table}[htbp!]
  \centering
  \fontsize{10}{10}
  \begin{threeparttable}
  \caption{Errors and convergence orders of finite element solution $\widetilde{\boldsymbol{u}}^n_h,\boldsymbol{u}^n_h$ with $P^1$-$P^1$ pair.}\label{FEVeloConvOrder}
    \begin{tabular}{c|c|c|c|c|c|c}
    \toprule
    \multirow{2}{*}{1/h} & \multicolumn{2}{c}{$\max\limits_{n}\left\|\boldsymbol{u}^n - \widetilde{\boldsymbol{u}}^n_h\right\|_{L^2}$} & \multicolumn{2}{c}{$\max\limits_{n}\left\|\boldsymbol{u}^n - \boldsymbol{u}^n_h\right\|_{L^2}$} & \multicolumn{2}{c}{$\left\|\nabla(\boldsymbol{u} - \widetilde{\boldsymbol{u}}_h)\right\|_{\ell^2(L^2)}$}  \cr
    \cmidrule(lr){2-3} \cmidrule(lr){4-5} \cmidrule(lr){6-7}
    & error & rate & error & rate & error & rate \cr
    \midrule
    4    & \quad5.3013e$-$01$\quad$  & \quad-$\quad$     & \quad4.6509e$-$01 $\quad$   & \quad- $\quad$    & \quad4.8187e$+$00 $\quad$ & \quad- $\quad$        \cr
    8    & \quad1.6490e$-$01$\quad$  & 1.7078   & 1.4931e$-$01   & 1.6392  & 2.6626e$+$00 & 0.85582  \cr
    16   & 4.3368e$-$02              & 1.9259   & 4.0108e$-$02   & 1.8963  & 1.3785e$+$00 & 0.94976  \cr
    32   & 1.0969e$-$02              & 1.9828   & 1.0527e$-$02   & 1.9298  & 7.1098e$-$01 & 0.95516  \cr
    64   & 2.7499e$-$03              & 1.9960   & 2.8273e$-$03   & 1.8966  & 3.7409e$-$01 & 0.92642  \cr
    \bottomrule
    \end{tabular}
    \end{threeparttable}
\end{table}
\begin{table}[htbp!]
  \centering
  \fontsize{10}{10}
  \begin{threeparttable}
  \caption{Errors and convergence orders of finite element pressure $p^n_h$ with $P^1$-$P^1$ pair.}\label{FEPressConvOrder}
    \begin{tabular}{c|c|c|c|c|c|c}
    \toprule
    \multirow{2}{*}{1/h} & \multicolumn{2}{c}{$\max\limits_{n}\|p^n - p^n_h\|_{L^2}$} & \multicolumn{2}{c}{$\left\|p - p_h\right\|_{\ell^2(L^2)}$} & \multicolumn{2}{c}{$\sqrt{\Delta t}\left\|\nabla(p-p_h)\right\|_{\ell^2(L^2)}$}  \cr
    \cmidrule(lr){2-3} \cmidrule(lr){4-5} \cmidrule(lr){6-7}
    & error & rate & error & rate & error & rate  \cr
    \midrule
    4    &  \quad 2.7636e$+$00 $\quad$  &  \quad - $\quad$  & \quad 2.2987e$+$00 $\quad$  &  \quad - $\quad$   &  \quad 1.4286e$+$00 $\quad$  &  \quad - $\quad$ \cr
    8    & 1.1144e$+$00  & 1.3103   & 8.8892e$-$01  & 1.3707 & 4.6814e$-$01 & 1.6096   \cr
    16   & 3.6664e$-$01  & 1.6038   & 2.7275e$-$01  & 1.7045 & 1.3827e$-$01 & 1.7595   \cr
    32   & 1.2463e$-$01  & 1.5567   & 8.1260e$-$02  & 1.7470 & 4.5158e$-$01 & 1.6144   \cr
    64   & 4.6335e$-$02  & 1.4275   & 2.5152e$-$02  & 1.6919 & 1.5553e$-$01 & 1.5378   \cr
    \bottomrule
    \end{tabular}
    \end{threeparttable}
\end{table}
\subsection{Convergence of projection POD}
After having the snapshots spaces, we determine $d_{\widetilde{u}}=\text{rank}(\widetilde{\mathcal{U}})=20$(for which $\gamma_i<10^{-15},i>d_{\widetilde{u}}$), $d_p=\text{rank}(\mathcal{P})=20$(for which $\epsilon_i<10^{-13},i>d_p$). Figure \ref{POD_EigenvaluesandEnergy_up} shows the decay of POD eigenvalues(left) of velocity $\gamma_i,i=1,\cdots,d_{\widetilde{u}}$ and pressure $\epsilon_i,i=1,\cdots,d_{p}$, together with the corresponding captured system's energy(right) in the form of $100\sum^{r}_{i=1}\gamma_i/\sum^{d_{\widetilde{u}}}_{i=1}\gamma_i$ for velocity and $100\sum^{r}_{i=1}\epsilon_i/\sum^{d_p}_{i=1}\epsilon_i$ for pressure. We note that the first $r=4$ POD modes already capture more than $99.99\%$ of the system's velocity-pressure energy.\\
\indent In order to confirm the computational efficiency advantage of the POD-ROM scheme over the finite element FOM scheme, we take the number of POD modes is $r=4$, whereas the number of degree of freedoms(DOFs) for velocity and pressure in finite element scheme are 8450 and 4225. We show in Table \ref{ComparisonofFEPOD} the comparison of cumulative time spent by the two schemes to run to some certain time in the process of time-stepping iterative, and the $L^2$ numerical spatial discrete error of velocity and pressure in two schemes at that time. We can see from Table \ref{ComparisonofFEPOD} that, compared with the finite element FOM scheme, POD-ROM scheme can effectively improve computational efficiency in the context of only consuming approximately 1/6 times to get the numerical solution with even higher accuracy.\\
\indent Figure \ref{VeloPressRelaErrorDesc} plots the temporal evolution of the discrete $L^2$ relative error(in semilogarithmic scale) of the reduced-order velocity and pressure with respect to the full order ones:
$\|\widetilde{\boldsymbol{u}}^n_h-\widetilde{\boldsymbol{u}}^n_r\|_{L^2}/\|\widetilde{\boldsymbol{u}}^n_h\|_{L^2}$ and $\|p^n_h – p^n_r\|_{L^2}/\|p^n_h\|_{L^2}$ for different number of POD modes. As expected in Theorem \ref{POD-ConvAnaly}, the errors would decrease as the number of POD modes $r$ increasing. For velocity, the left figure in Figure \ref{VeloPressRelaErrorDesc} numerically demonstrate this result, and for pressure, the error would increase slightly when $r$ from $6$ to $8$, but when we continue to increase $r$, the error will decrease, which is consistent with the Theorem \ref{POD-ConvAnaly} as a whole.
\begin{figure}
\centering
\includegraphics[width=13.125cm,height=9.84375cm]{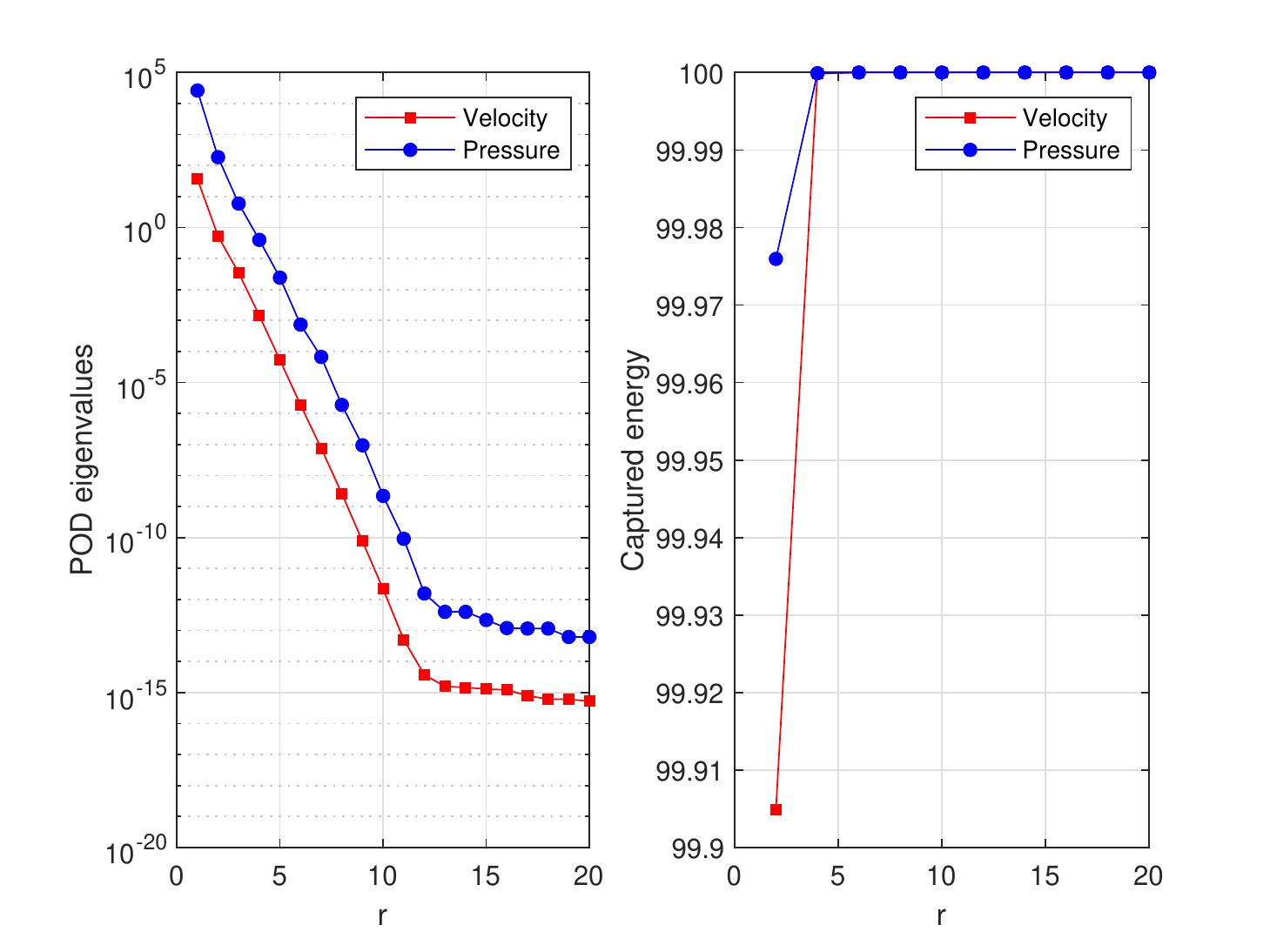}
\caption{POD velocity-pressure eigenvalues(left) and captured system's velocity-pressure energy(right).}\label{POD_EigenvaluesandEnergy_up}
\end{figure}

\begin{table}[htbp!]
  \centering
  \fontsize{10}{10}
  \begin{threeparttable}
  \caption{Comparison of errors and time consumed by time-stepping iterative between finite element scheme with $P^1$-$P^1$ pair and POD scheme with $r=4$ on the mesh of $h=1/64$ and $\Delta t=0.1h^2$.}\label{ComparisonofFEPOD}
    \begin{tabular}{ccccccc}
    \toprule
    \multirow{2}{*}{n} & \multicolumn{3}{c}{finite element scheme} & \multicolumn{3}{c}{POD scheme}  \cr
    \cmidrule(lr){2-4} \cmidrule(lr){5-7}
    & $\|\boldsymbol{u}^n \!-\! \widetilde{\boldsymbol{u}}^n_h\|_{L^2}$ & $\|p^n \!-\! p^n_h\|_{L^2}$ & CPU run time(s) & $\|\boldsymbol{u}^n \!-\! \widetilde{\boldsymbol{u}}^n_r\|_{L^2}$ & $\|p^n \!-\! p^n_r\|_{L^2}$ & CPU run time(s) \cr
    \midrule
    2500  & 2.3789e$-$03 & 2.9458e$-$02 & 689   & 2.2860e$-$03 & 2.7823e$-$02 & 113       \cr
    5000  & 2.3929e$-$03 & 2.9253e$-$02 & 1338  & 2.3007e$-$03 & 2.7642e$-$02 & 227       \cr
    7500  & 2.3740e$-$03 & 2.8975e$-$02 & 2005  & 2.2826e$-$03 & 2.7379e$-$02 & 339       \cr
    10000 & 2.3452e$-$03 & 2.8591e$-$02 & 2668  & 2.2549e$-$03 & 2.7015e$-$02 & 452       \cr
    20000 & 2.1443e$-$03 & 2.6010e$-$02 & 5339  & 2.0618e$-$03 & 2.4573e$-$02 & 914       \cr
    30000 & 1.8163e$-$03 & 2.1886e$-$02 & 8060  & 1.7464e$-$03 & 2.0673e$-$02 & 1399      \cr
    40000 & 1.3805e$-$03 & 1.6464e$-$02 & 10746 & 1.3274e$-$03 & 1.5547e$-$02 & 1900      \cr
    \bottomrule
    \end{tabular}
    \end{threeparttable}
\end{table}

\begin{figure}
\centering
\includegraphics[width=13.125cm,height=9.84375cm]{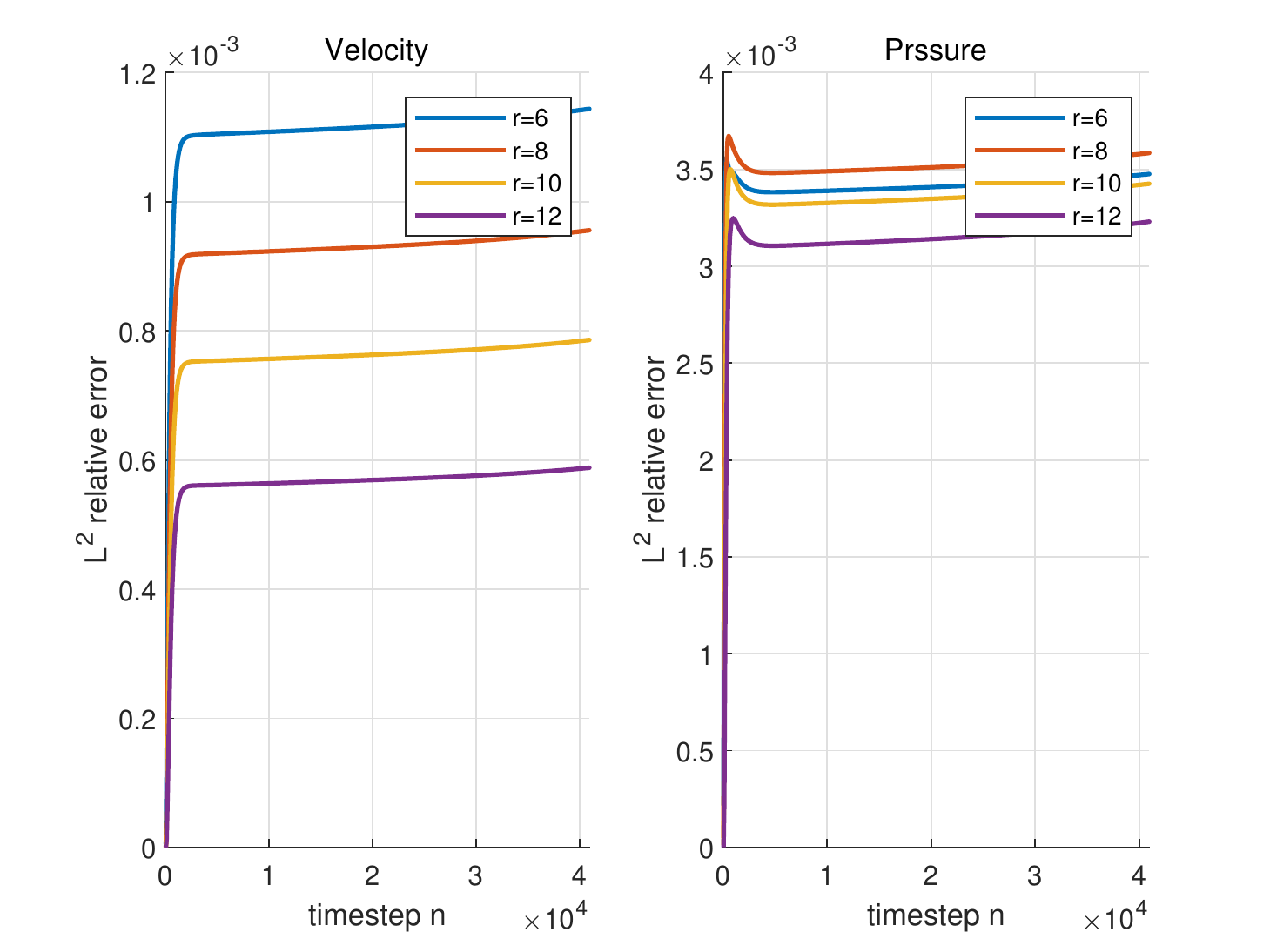}
\caption{Temporal evolution of $L^2$ relative error of velocity $\boldsymbol{u}^n_{r}$ with respect to $\boldsymbol{u}^n_{h}$(left) and velocity $p^n_{r}$ with respect to $p^n_{h}$(right) .}\label{VeloPressRelaErrorDesc}
\end{figure}



\section{Conclusions}
\noindent In this paper, we proposed an efficient projection POD-ROM, which combined the advantages of classical projection method and POD technique.\\
\indent The main contribution of the present paper consisted of two aspects: the first one was high computational efficiency. Through auxiliary intermediate velocity variable, the classical projection method decoupled the velocity variable and pressure variable, meanwhile decoupled the saddle-point system arose from Stokes equations, so one strength of projection method lied in its high computational efficiency, or low computational costs; Furthermore, POD technique was utilized to get the ROM, which have made the newly proposed projection POD-ROM had high computational efficiency. The second contribution was based on the fact that, in the fully discrete scheme of the classical projection method, the original scheme could be rewritten into a PSPG-type pressure stabilization scheme by eliminating the end-of-step velocity, where the pressure stabilized term $\Delta t(\nabla p^{n+1}_h, \nabla q^{n+1}_h)$, $\Delta t=\mathcal{O}(h^2)$, was inherent, so that some flexible mixed finite element spaces pairs(for example, $P^1$-$P^1$ pair) could be used without considering the classical LBB/inf-sup condition for mixed POD spaces, which was different from other stabilized FE-POD-ROM to overcome LBB/inf-sup condition by adding extra stabilization terms. \\
\indent Numerical experiments have been conducted to confirm the convergence for both projection finite element scheme and projection POD scheme. We first numerically confirmed the PSPG-type classical projection owned the desired convergence orders consistent with theoretical analysis, and then, finite element solutions were taken as the snapshots to formulate POD bases/modes which are used to get reduced-order velocity and pressure variables, so in theory, the discrete error between exact solutions and reduced-order solutions should converge to zero as increasing the number of POD modes, we numerically verify this result. Apart from the projection POD-ROM convergence, we also conduct the experiment to compare the error and computational time between projection finite element FOM and projection POD-ROM, and the results revealed the fact that projection POD-ROM not only had less discretization error, but also less computational costs, compared with projection finite element FOM.  \\
\indent One future research direction will be the applied of projection POD on nonlinear nonstationary Navier-Stokes equations. Another investigation is validation of fulfillment of LBB/inf-sup condition for mixed POD spaces.

\section*{Acknowledgment}
\noindent The authors thank the anonymous referees for their constructive comments and suggestions which improved the manuscript.

\bibliographystyle{plain}
\bibliography{Reference}
\end{document}